\documentclass[review]{siamart171218}

\usepackage{amsfonts}
\usepackage{bbm}
\usepackage{mathrsfs}
\usepackage{lipsum}
\usepackage{geometry}
\usepackage{multirow}
\usepackage{graphics,graphicx,epsf,subfigure,epstopdf,epsfig}
\usepackage{amssymb}
\usepackage{amsmath}
\usepackage{float}
\usepackage{boxedminipage}
\usepackage{caption}
\usepackage{enumerate}
\usepackage{enumitem}
\usepackage{accents}
\usepackage{fancybox}
\usepackage{slashbox}
\usepackage{bm}
\usepackage{algorithm}
\usepackage{algorithmic}

\ifpdf
  \DeclareGraphicsExtensions{.eps,.pdf,.png,.jpg}
\else
  \DeclareGraphicsExtensions{.eps}
\fi

\numberwithin{theorem}{section}

\newcommand{\TheTitle}{Anderson Acceleration for Nonsmooth Fixed Point Problems}
\newcommand{\TheAuthors}{}

\headers{}{\TheAuthors}

\title{{\TheTitle}\thanks{Submitted to the editors XX.}
}

\author{
  Wei Bian\thanks{Department of Mathematics, Harbin Institute of Technology, Harbin, China (\email{bianweilvse520@163.com}). The work of this author was partially supported by the NSF
foundation (11871178,61773136) of China .}
  \and
  Xiaojun Chen\thanks{Department of
Applied Mathematics, The Hong Kong Polytechnic University, Hong Kong, China (\email{maxjchen@polyu.edu.hk}). The work of this author
was partially supported by Hong Kong Research Grant Council grant
(15300210) }
}

\usepackage{amsopn}

\ifpdf
\hypersetup{
  pdftitle={\TheTitle},
  pdfauthor={\TheAuthors}
}
\fi

\newtheorem{remark}{\bf Remark}[section]
\newtheorem{example}{\bf Example}[section]
\newtheorem{assumption}{Assumption}[section]
\begin{document}

\maketitle

\begin{abstract}
We give new convergence results of Anderson acceleration for the composite $\max$ fixed
point problem. We prove that Anderson(1) and EDIIS(1) are q-linear convergent with a smaller
q-factor
than existing q-factors.
Moreover, we propose a smoothing approximation of the composite max function in the contractive fixed point problem.
We show that the smoothing approximation is a contraction mapping with
the same fixed point as the composite $\max$ fixed point
problem. Our results rigorously confirm  that the nonsmoothness does not affect the convergence rate
of Anderson acceleration method when we use the proposed smoothing approximation for
the composite $\max$ fixed
point problem.
Numerical results for constrained minimax problems, complementarity problems and nonsmooth
differential equations are presented to show the efficiency and good performance of the proposed Anderson acceleration method with smoothing approximation. \end{abstract}


\section{Introduction}
In this paper, we focus on the convergence analysis of Anderson(m) and EDIIS(m) for the following composite max fixed point problem
\begin{equation}\label{eq1-0}
u=G(u):=H(P_{\Omega}(Q(u))),
\end{equation}
where $H:\mathbb{R}^l\rightarrow\mathbb{R}^n$ and $Q:\mathbb{R}^n\rightarrow \mathbb{R}^l$ are
Lipschitz continuously differentiable functions, $\Omega$ is a box subset of $\mathbb{R}^l$, and $P_{\Omega}$ is the projection on $\Omega$. Problem (\ref{eq1-0}) arises from many applications in engineering and finance including
minimax problems, complementarity problems, nonsmooth integral equations and nonsmooth
differential equations.

Anderson acceleration was originally introduced in the context of integral equations by Anderson in 1965 \cite{Anderson}. It is a class of methods for solving the fixed point problem
$u=G(u)$, where $G$ is a continuous function from $D\subseteq \mathbb{R}^n$ to $D$,
and uses a history of search directions to improve the convergence rate of
the fixed point method
\begin{equation}\label{Picard}
u_{k+1}=G(u_k).
\end{equation}
Anderson acceleration method has been widely used in electronic structure computation \cite{Anderson,ChenKelley2015,Fang,Schneider,
TothKelley2015,Walker}, chemistry and physics \cite{An,Stasiak}, and specific optimization problems \cite{NN,Walker}. In particular, Anderson acceleration is designed to solve the fixed point problem when computing
the Jacobian of $G$ is impossible or too costly.
Anderson acceleration is also known as the Pilay mixing \cite{Pulay}, DIIS (direct inversion on iterative subspace) \cite{Kudin,Lin,Stasiak}, nonlinear GMRES method \cite{Carlson,Miller,Washio}, and interface quasi-Newton \cite{Eyert,Haelterman}.
A formal description of Anderson acceleration is
presented in Algorithm \ref{algo-o} and often called Anderson(m).
\begin{algorithm}
\caption{Anderson(m)}\label{algo-o}
Choose $u_0\in D$ and a positive integer $m$. Set $u_1=G(u_0)$ and $F_0=G(u_0)-u_0$. \\
\textbf{for $k=1,2,...$ do}

\qquad  set $F_k=G(u_k)-u_k$;

\qquad choose $m_k=\min\{m,k\}$;

\qquad solve
\begin{equation}\label{Anderson_alpha}
\min\,\,\left\|\sum\nolimits_{j=0}^{m_k}\alpha_j{F}_{k-m_k+j}\right\|\quad \quad
\mbox{s.t.}\,\sum\nolimits_{j=0}^{m_k}\alpha_j=1
\end{equation}
\qquad to find a solution $\{\alpha_j^k:j=0,\ldots,m_k\}$, and set
\begin{equation}\label{u_k+1}
u_{k+1}=\sum\nolimits_{j=0}^{m_k}\alpha_j^k{G}(u_{k-m_k+j}).
\end{equation}
\textbf{end for}
\end{algorithm}

Anderson(m) maintains a history of function values of $G(\cdot)$ at $u_{k-m_k+j}$, $j=0,\ldots, m_k$,
where $m_k$ is an algorithmic parameter that indicates the depth of the accelerated Anderson iterations.
Using these function values, Anderson(m) defines a new iterate by a linear combination of the last {$m_k+1$} iterates, where the coefficients of the linear combination are computed at each iteration by the convex optimization problem in \eqref{Anderson_alpha}.
When $m=0$, Anderson acceleration is the fixed point method in (\ref{Picard}),
which is also known as the Picard method. In practice, each $m_k$ may be different to maintain the acceptable conditioning of $(F_{k-m_k+j})_{j=0}^{m_k}$ \cite{Walker} and can be dynamically updated to improve the performance \cite{Pollock2}. The optimization problem \eqref{Anderson_alpha} in Anderson(m) does not specify the norm in its general form and using different norms will not affect the convergence \cite{TothKelley2015}. Throughout this paper, we consider problem \eqref{Anderson_alpha} in the sense of Euclidean norm.
Notice that the description of Anderson(m) in Algorithm \ref{algo-o} is convenient for analysis, but the readers may refer to \cite{TothKelley2015,Walker} and references therein for its efficient implementation.

The EDIIS(m) \cite{Kudin} differs from Anderson(m) by adding nonnegativity  constraints in (\ref{Anderson_alpha}),
that is, replacing (\ref{Anderson_alpha}) by the following minimization problem
$$
\min\,\,\left\|\sum_{j=0}^{m_k}\alpha_j{F}_{k-m_k+j}\right\|\quad
\mbox{s.t.}\,\sum_{j=0}^{m_k}\alpha_j=1, \quad \alpha_j\ge 0,\,  j=0,\ldots, m_k.
$$

Suppose $G:D\rightarrow D$ is a contraction mapping with factor
$c\in(0,1)$ in the Euclidean norm $\|\cdot\|$ on a closed set $D\subset\mathbb{R}^n$, that is,
$$
\|G(u)-G(v)\|\le c \|u-v\|, \; \forall \, u,v \in D.$$
By the contraction mapping theorem \cite{Ortega}, $G$ has
 a unique fixed point $u^*\in D$, which is the unique solution of the system of nonlinear equations
$$F(u):=G(u)-u=0.$$
Without loss of generality, we assume that there is $\mathcal{B}(\delta,u^*):=\{u \in \mathbb{R}^n :  \, \|u-u^*\| \le \delta\}\subset D$ with $\delta>0$.
For a contraction mapping $G$, it is known that the fixed-point method in \eqref{Picard} has q-linear convergence rate, that is
$
\|u_{k+1}-u^*\|\le c\|u_k-u^*\|
$
holds in ${\cal B}(\delta,u^*)$. However, the theoretical convergence analysis of Anderson(m) had not been proved for a long time after it being brought forward and widely used. The first mathematical convergence result for Anderson(m) was given by Toth and Kelley in 2015 \cite{TothKelley2015}.
Under the assumption that $G$ is Lipschitz continuously differentiable in $D$,
Toth and Kelley \cite{TothKelley2015} showed the r-linear convergence of Anderson(m) with r-factor
$\hat{c}\in (c,1)$ as follows,
$$\|F(u_k)\|\le \hat{c}^k \|F(u_0)\| \quad {\rm and} \quad
\|u_k-u^*\|\le \left(\frac{1+c}{1-c}\right)\hat{c}^k\|u_0-u^*\|.$$
Without the differentiability of $G$, Chen and Kelley \cite{ChenKelley2015} showed the r-linear
convergence of EDIIS(m) with r-factor $\hat{c}=c^{1/(m+1)}$ as follows,
$$\|u_k-u^*\|\le \hat{c}^k\|u_0-u^*\|.$$
Moreover, Bian, Chen and Kelley \cite{BCK} showed the q-linear convergence of Anderson(1) and EDIIS(1) with
q-factor $(3c-c^2)/(1-c)$ for general nonsmooth fixed point problems in a Hilbert space, and
r-linear convergence of Anderson(m) and EDIIS(m) with r-factor
$\hat{c}\in (c,1)$ for a class of integral equations in which the operator can be written as
the sum of a smooth term and a nonsmooth term having a sufficiently small Lipschitz constant.
Zhang et al \cite{Zhang} proposed a globally convergent variant of Anderson acceleration for nonsmooth fixed point problems, but did not provide a rate of convergence. The first mathematical view to show the superiority of local convergence of Anderson method for
the discretizations of the steady Navier-Stokes equations was proved by Pollock, Rebholz and Xiao in \cite{Pollock}. And the similar idea was extended to a more general fixed-point iterations by
Evans, Pollock, Rebholz, and Xiao \cite{Evans}.
Most recently, Pollock and Rebholz \cite{Pollock2} showed a novel one-step bound of Anderson method with a more general acceleration iteration,  which not only sharpens the convergence results for contractive mapping in \cite{Evans}, but also explains some mechanism of Anderson acceleration for noncontractive cases.
Overall, Anderson acceleration can significantly improve the computational performance
of the fixed point method in practice.
We refer the readers to \cite{Evans,Fang,Pollock2,Walker} and references therein for detailed discussions on its research history and practical applications.

Throughout this paper, we suppose $\Omega$ is defined by
\begin{equation}\label{eq-omiga}
\Omega=\{ w \in  \mathbb{R}^l \, |\, \underline{w}\le w\le \overline{w}\}
\end{equation}
with
$\underline{w} \in \{\{-\infty\}\cup \mathbb{R}\}^{l}$, $\overline{w}\in\{\{\infty\}\cup\mathbb{R}\}^{l}$
and $\underline{w}<\overline{w}$. Then,
$P_{\Omega}$ can be expressed by the following composite max form
\begin{equation}\label{eq-P1}
P_{\Omega}(w)={\rm argmin}_{v\in \Omega} \|v-w\|^2 =\max\{\underline{w}-w,0\}+{w}-\max\{w-\overline{w},0\},
\end{equation}
where ``max'' means componentwise. The formulation of $P_{\Omega}$ in \eqref{eq-P1} will play a key role in the analysis of this paper.
Here we declare that $(-\infty)-a=-\infty$ and $a-(\infty)=-\infty$ for any
$a\in\mathbb{R}$.
When $\Omega=\mathbb{R}^l_+:=\{w \in  \mathbb{R}^l \, |\, w\ge 0\}$, the expression of $P_{\Omega}$ in \eqref{eq-P1}
is reduced to
$$P_{\Omega}(w)=\max\{-w,0\}+{w}.$$
In particular, if $\underline{w}_i=-\infty$ and $\overline{w}_i=\infty$ for all $i\in \{1,\ldots,l\}$,
then $G=H(Q(u))$ is Lipschitz continuously differentiable on $D$, which is
the case considered in \cite{TothKelley2015}. Thus,
we focus on the case that there is at least an $i\in \{1,\ldots,l\}$ such that
$-\infty< \underline{w}_i$ or $\overline{w}_i<\infty$, which means that $G$ is nonsmooth on
$D$ in general.

The contributions of this paper are new convergence results of Anderson acceleration method
for composite max fixed point problem (\ref{eq1-0}). In section \ref{section2},
we prove that Anderson(1) and EDIIS(1) are q-linear convergent for problem (\ref{eq1-0})
with q-factor $\hat{c}\in (\frac{2c-c^2}{1-c}, 1)$, which can be strictly smaller than the existing q-factor $(3c-c^2)/(1-c)$ proved in \cite{BCK,TothKelley2015}. In section \ref{section3}, we give the
contraction consistent properties between $G$ and its smoothing approximations. Then,
we propose a new smoothing approximation
${\cal G}(\cdot, \mu)$ of $G$. We show that there is $\bar{\mu}>0$, such that ${\cal G}(\cdot,\mu)$ is continuously differentiable, contractive on $D$, and
$u^*={\cal G}(u^*, \mu)=G(u^*)$, for any fixed $\mu \in (0, \bar{\mu}]$.
To improve the ability and performance of Anderson acceleration method
for solving problem (\ref{eq1-0}), we propose a smoothing Anderson acceleration (s-Anderson(m)) in Algorithm \ref{algo-o1} with the proposed smoothing function of $G$ and updating scheme for smoothing parameters. We prove that s-Anderson(m) for (\ref{eq1-0}) owns the same r-linear convergence rate as Anderson(m) for continuously differentiable problems.
In section \ref{section4}, we use numerical examples from constrained minimax problems, pricing American options and nonsmooth Dirichlet problem to illustrate our theoretical results. Preliminary
numerical results show that s-Anderson(m) can efficiently solve the nonsmooth fixed point problem (\ref{eq1-0}) and outperform
Anderson(m) in most cases.
\section{q-linear convergence of Anderson(1) and EDIIS(1)}\label{section2}
For $m_k=1$, the optimal solution of problem (\ref{Anderson_alpha}) owns the closed form
$(1-\alpha_k, \alpha_k)^{\rm T}$ with
\begin{equation}\label{eq-alpha}
\alpha_k=\frac{F_k^{\rm T}({F}_k-{F}_{k-1})}{\|F_k-F_{k-1}\|^2}
\end{equation}
and the iterate can be expressed as
\begin{equation}\label{eq-u1}
u_{k+1}=(1-\alpha_k){G}(u_k)+\alpha_k{G}(u_{k-1}).
\end{equation}

In the remainder of this paper, we need the following assumption.
\begin{assumption}\label{ass-Q-H}
Functions $Q$ and  $H$ in (\ref{eq1-0}) satisfy the following conditions.
\begin{itemize}
\item [(i)] $Q$ is Lipschitz continuously differentiable on
$D$ with Lipschitz constant $c_Q$.
\item [(ii)] $H$ is Lipschitz continuously differentiable on an open set $D_H$ containing $\Omega$ as a subset with Lipschitz constant $c_H$.
\item [(iii)] $c:=c_Hc_Q<1$.
\end{itemize}
\end{assumption}

Note that the Lipschitz continuous differentiability of $Q$ and $H$ cannot imply the
differentiability of $G$ on $D$ due to the existence of projection operator $P_{\Omega}$
in its formulation. Since $P_{\Omega}$ is Lipschitz continuous with Lipschitz constant $1$, from
 \[\begin{aligned}
 \|H(P_{\Omega}(Q(u))) - H(P_{\Omega}(Q(v)))\|\leq &c_H\|P_{\Omega}(Q(u))-P_{\Omega}(Q(v))\|\\
 \le& c_H\|Q(u) - Q(v)\|
 \leq c_Hc_Q\|u-v\|,
 \end{aligned}\]
we find that $G$ in (\ref{eq1-0}) is a contraction mapping on $D$ with
factor $c=c_Hc_Q$ under Assumption \ref{ass-Q-H}. Then, it gives
\begin{equation}\label{smoothing3}
(1-c)\|u-u^*\|\leq\|F(u)\|\leq(1+c)\|u-u^*\|,\quad \forall u\in D.
\end{equation}

The following theorem shows that the local $q$-linear convergence factor of
Anderson(1) and EDIIS(1) can be improved to any $\hat{c}\in(\frac{2c-c^2}{1-c}, 1)$ for
(\ref{eq1-0}), which can be
 strictly smaller than the factor $\frac{\;3c-c^2\;}{\;1-c\;}$ given in \cite{BCK,TothKelley2015}.
\begin{theorem}\label{theorem2}
 Let $\{u_k\}$ be the sequence generated by Anderson(1) for (\ref{eq1-0}). Suppose Assumption
 \ref{ass-Q-H} holds and
$\bar{c}=\frac{2{c}-{c}^2}{1-{c}}<1$.
For any $\hat{c}\in\left(\bar{c}, 1\right)$,
if $u_0$ is sufficiently close to $u^*$, then $\{u_k\}$
 converges to $u^*$ q-linearly with factor $\hat{c}$, i.e.
 \begin{equation}\label{eq6-2}
\|{F}(u_{k+1})\|\le \hat{c}\|{F}(u_{k})\|,\quad k=0,1,\ldots.
\end{equation}
\end{theorem}
\begin{proof}
Give $\epsilon>0$. Reduce $\delta>0$ if necessary such that $\delta\leq\epsilon$ and $\mathcal{B}(\delta, u^*)\subseteq D$.
Since $c\leq\bar{c}<\hat{c}$, (\ref{eq6-2}) is trivially true for $k=0$. Then, we prove
(\ref{eq6-2}) by induction and assume it holds for $0\leq k\leq K$.
Let
$$
0<\varrho\leq\min\{1,(\overline{{w}}_i-\underline{w}_i)/3: i=1,2,\ldots,l\}.
$$
Here we declare that $\overline{{w}}_i-\underline{w}_i =\infty$ if  $\overline{{w}}_i=\infty$ or/and  $\underline{w}_i=-\infty$.

By \eqref{eq-alpha}, we have
\begin{equation}\label{eq3}
|\alpha_k|\leq\frac{\|F(u_k)\|}{\|F(u_k)-F(u_{k-1})\|}\quad\mbox{and}\quad
|1-\alpha_k|\leq\frac{\|F(u_{k-1})\|}{\|F(u_k)-F(u_{k-1})\|},\,\,\forall k.
\end{equation}

Similar to the analysis in \cite[Theorem 1.3]{BCK} and
 by the hypothesis in \eqref{eq6-2} for $0\leq k\leq K$, we have that
\begin{equation}\label{add1}
|\alpha_k|\leq\frac{\hat{c}}{1-\hat{c}}\quad\mbox{and}\quad
\|u_k-u^*\|\leq\frac{\|F(u_k)\|}{1-c}\leq\frac{\hat{c}^k(1+c)\|u_0-u^*\|}{1-c}.
\end{equation}
Then,
$$\begin{aligned}
\|u_{K+1}-u^*\|
=&\|(1-\alpha_K){G}(u_K)+\alpha_K{G}(u_{K-1})-G(u^*)\|\\
\leq &c|1-\alpha_K|\|u_K-u^*\|+c|\alpha_K|\|u_{K-1}-u^*\|\leq\frac{2c(1+c)\hat{c}^K}{(1-c)(1-\hat{c})}\|u_0-u^*\|.
\end{aligned}$$
Similarly,
$$
\begin{aligned}
\|(1-\alpha_K)u_K+\alpha_Ku_{K-1}-u^*\|\leq\frac{2(1+c)\hat{c}^K}{(1-c)(1-\hat{c})}\|u_0-u^*\|.
\end{aligned}$$
Thus,
there exists
$\delta_0>0$ such that if $u_0\in\mathcal{B}(\delta_0, u^*)$,  then
\[u_{k}\in\mathcal{B}(\delta,u^*),\, k=0,\ldots,K+1\quad {\rm and} \quad \,
(1-\alpha_K)u_K+\alpha_Ku_{K-1}\in\mathcal{B}(\delta,u^*).\]

Now we estimate $\|F(u_{K+1})\|$ by using
\begin{equation}\label{eqF}
\|F(u_{K+1})\|=\|G(u_{K+1})-u_{K+1}\|
\leq\|A_K\|+\|B_K\|,\end{equation}
where
\[A_K=G(u_{K+1})-G((1-\alpha_K)u_{K}+\alpha_Ku_{K-1}),\quad B_K=G((1-\alpha_K)u_{K}+
\alpha_Ku_{K-1}) -u_{K+1}.\]

The estimate of $A_k$ is straightforward as it is in \cite{BCK,ChenKelley2015,TothKelley2015},
which gives
\begin{equation}\label{eq-A1}
\|A_K\|\leq {c}\|(1-\alpha_K)(G(u_K)-u_K) +\alpha_K(G(u_{K-1})-u_{K-1})\|\le {c}\|F(u_K)\|.
\end{equation}
Now, we estimate $\|B_K\|$.  First, we note that
$\psi(t)=\left\{\begin{aligned}
&\max\{0,t\} &&\mbox{if $|t|>\varrho$}\\
&\frac{(t+\varrho)^2}{4\varrho} &&\mbox{if $|t|\leq \varrho$}
\end{aligned}\right.$ is a smoothing approximation of $\max\{t,0\}$.
Then, by \eqref{eq-P1},
$$\Phi(w)=\Psi(\underline{w}-w)+{w}-\Psi(w-\overline{w})$$
is a smoothing approximation of $P_{\Omega}(w)$, where
$$
\Psi(v)=(\psi(v_1),\ldots,\psi(v_{l}))^{\rm T}.$$
By virtue of the value of
$\varrho$, for any $i\in\{1,\ldots,l\}$ and $w_i\in\mathbb{R}$, at most one of $|\underline{w}_i-w_i|\leq\varrho$
and $|w_i-\overline{w}_i|\leq\varrho$ holds.
Then, since $|\psi'(t)|\leq1$, $\forall t\in\mathbb{R}$, for any $w,\tilde{w}\in\mathbb{R}^{l}$,
we obtain
\begin{equation}\label{eq-pro1}
\|\Phi(w)-\Phi(\tilde{w})\|\leq 2\|w-\tilde{w}\|.
\end{equation}

Next, recalling the definition of $\psi$, we have
\[\max\{t,0\}-\psi(t)=
\left\{\begin{aligned}
&0\quad &&\mbox{if $|t|>\varrho$}\\
&-{(\varrho-|t|)^2}/{4\varrho}&&\mbox{if $|t|\leq\varrho$},
\end{aligned}\right.\]
which implies the absolute value and Lipschitz constant of $\max\{t,0\}-\psi(t)$ on $\mathbb{R}$
are upper bounded by ${\varrho}/{4}$ and ${1}/{2}$, respectively. Then, for any $w,\tilde{w}\in\mathbb{R}^{l}$,
we have
\begin{eqnarray}
&\|P_{\Omega}(w)-\Phi(w)\|\leq{\sqrt{l}}\varrho/4,\label{eq-pro3}\\
&\|P_{\Omega}(w)-\Phi(w)-(P_{\Omega}(\tilde{w})-\Phi(\tilde{w}))\|\leq\frac{1}{2}\|w-\tilde{w}\|.\label{eq-pro2}
\end{eqnarray}
Denote
\[G_S(u)=H(\Phi(Q(u)))\quad\mbox{and}\quad G_N(u)=G(u)-G_S(u).\]

Then from the definition of $u_{K+1}$ in (\ref{eq-u1}), we have
\begin{equation}\label{eq2}
\|B_{K}\|\leq\|M_K\|+\|N_K\|,
\end{equation}
with
\[M_K=G_S((1-\alpha_K)u_{K}+\alpha_Ku_{K-1})-(1-\alpha_K)G_S(u_{K})-\alpha_KG_S(u_{K-1})\]
and
\[N_K=G_N((1-\alpha_K)u_{K}+\alpha_Ku_{K-1})-(1-\alpha_K)G_N(u_{K})-\alpha_KG_N(u_{K-1}).\]

Notice that $\psi$ is Lipschitz continuously differentiable on $\mathbb{R}$.
By the Lipschitz continuous differentiability of $Q$ and $H$, $G_S$ is Lipschitz continuously differentiable on $\mathcal{B}(\delta, u^*)$,
which inspires us to estimate $M_k$ exactly by the same way as in \cite[Corollary 2.5]{TothKelley2015} to get
\begin{equation}\label{eq26-2}
\begin{aligned}
\|M_k\|\leq\frac{\gamma|\alpha_K||1-\alpha_K|\|u_K-u_{K-1}\|^2}{2}\leq\frac{\gamma\|F(u_{K-1})\|}{2(1-{c})^2}\|F(u_K)\|
\leq\frac{\gamma(1+{c})\epsilon}{2(1-{c})^2}\|F(u_K)\|,
\end{aligned}\end{equation}
where $\gamma$ is the Lipschitz constant of $G_S'$ on $\mathcal{B}(\delta,u^*)$ and we use
$\|F(u_{K-1})\|\leq(1+{c})\|u_{K-1}-u^*\|\leq(1+{c})\delta\leq(1+{c})\epsilon$ in the last inequality.

The final stage of this proof is to evaluate $\|N_K\|$, which is the main part in this proof.

To do this, the
first thing is to evaluate the Lispchitz constant of $G_N$ around $u^*$.
For any $u,v\in\mathcal{B}(\delta,u^*)$, by the Lipschitz continuous differentiability of
$H$ and the mean value theorem for a vector-valued function, we have
\[\begin{aligned}
&\|G_N(u)-G_N(v)\|
=\|H(P_{\Omega}(Q(u)))-H(\Phi(Q(u)))-H(P_{\Omega}(Q(v)))+H(\Phi(Q(v)))\|\\
=&\left\|\left(\int_0^1H'(\hat{\xi}(t)){\rm d}t\right)(P_{\Omega}(Q(u))-P_{\Omega}(Q(v)))-\left(\int_0^1H'(\bar{\xi}(t)){\rm d}t\right)(\Phi(Q(u))-\Phi(Q(v)))\right\|,
\end{aligned}\]
where $\hat{\xi}(t)=tP_{\Omega}(Q(v))+(1-t)P_{\Omega}(Q(u))$ and
$\bar{\xi}(t)=t\Phi(Q(v))+(1-t)\Phi(Q(u))$.
Denote
\[G_N^1=\left(\int_0^1H'(\hat{\xi}(t)){\rm d}t\right)(P_{\Omega}(Q(u))-P_{\Omega}(Q(v)))-\left(\int_0^1H'(\bar{\xi}(t)){\rm d}t\right)(P_{\Omega}(Q(u))-P_{\Omega}(Q(v))),\]
\[G_N^2=\left(\int_0^1H'(\bar{\xi}(t)){\rm d}t\right)(P_{\Omega}(Q(u))-P_{\Omega}(Q(v)))-\left(\int_0^1H'(\bar{\xi}(t)){\rm d}t\right)(\Phi(Q(u))-\Phi(Q(v)),\]
then
\[\|G_N(u)-G_N(v)\|\leq\|G_N^1\|+\|G_N^2\|.\]

By \eqref{eq-pro1}, \eqref{eq-pro3}, and the definitions of $\hat{\xi}(t)$ and $\bar{\xi}(t)$, for any $t\in[0,1]$, it holds
\begin{equation}\label{eq10}
\|\hat{\xi}(t)-\bar{\xi}(t)\|
\leq\|P_{\Omega}(Q(v))-\Phi(Q(v))\|
+\|P_{\Omega}(Q(u))-\Phi(Q(u))\|
\leq{\;\sqrt{l}\varrho\;}/{\;2\;}.
\end{equation}
Due to the convexity of $\Omega$, $\hat{\xi}(t)\in\Omega$ for all $t\in[0,1]$. Then, by \eqref{eq10}, we can suppose $\bar{\xi}(t)\in D_H$ for all $t\in[0,1]$ by reducing $\varrho$ if necessary. Moreover, since $u,v\in\mathcal{B}(\delta, u^*)$, $\bar{\xi}(t)$ and $\hat{\xi}(t)$ are bounded for all $t\in[0,1]$.
Then, using the Lipschitz continuous differentiability of $H$, there exists $\theta>0$
such that it holds
$$\int_0^1\left\|H'(\hat{\xi}(t))-H'(\bar{\xi}(t))\right\|{\rm d}t\leq \theta\max_{0\leq t\leq1}\|\hat{\xi}(t)-\bar{\xi}(t)\|,
$$
combining which with \eqref{eq10} gives
\[
\|G_N^1\|\leq\left(\int_0^1\left\|H'(\hat{\xi}(t))-H'(\bar{\xi}(t))\right\|{\rm d}t\right)\left\|P_{\Omega}(Q(u))-P_{\Omega}(Q(v))\right\|
\leq\left(\frac{\;\sqrt{l}\varrho\theta c_Q\;}{\;2\;}\right)\|u-v\|.
\]
Thus, by reducing $\varrho$ if necessary, we obtain
\begin{equation}\label{th4-1}\|G_N^1\|\leq\epsilon\|u-v\|.\end{equation}

To evaluate $G_N^2$, by \eqref{eq-pro2} and $\bar{\xi}(t)\in D_H$ for all $t\in[0,1]$, we have
\begin{equation}\label{th4-2}
\begin{aligned}
\|G_N^2\|\leq&\left(\int_0^1\|H'(\bar{\xi}(t))\|{\rm d}t\right)\|P_{\Omega}(Q(u))-\Phi(Q(u))-(P_{\Omega}(Q(v))-\Phi(Q(v)))\|\\
\leq&\frac{1}{2}c_Hc_Q\|u-v\|=\frac{1}{2}c\|u-v\|.
\end{aligned}
\end{equation}
Hence,
\eqref{th4-2} together with (\ref{th4-1}) gives that the Lipschitz constant of $G_N$ around $u^*$ can be bounded by $\frac{1}{2}{c}+\epsilon$. Using it to $N_K$, we have
\begin{equation}\label{eqD}
\begin{aligned}
\|N_K\|=&\|G_N(u_K -\alpha_K(u_{K}-u_{K-1}))-G_N(u_{K})+\alpha_K(G_N(u_{K})-G_N(u_{K-1}))\|\\
\leq&(\frac{1}{2}c+\epsilon)2|\alpha_K|\|u_{K}-u_{K-1}\|.
\end{aligned}
\end{equation}

Then, \eqref{eq3} and (\ref{eqD}) imply
\begin{equation}\label{eq23}
\|N_K\|\leq\frac{c+2\epsilon}{1-{c}}\|F(u_K)\|.
\end{equation}

We obtain from (\ref{eqF}), \eqref{eq-A1}, \eqref{eq2}, \eqref{eq26-2} and (\ref{eq23}) that
\[
\|F(u_{K+1})\|\leq
\left(\bar{c}+\iota\epsilon\right)\|F(u_K)\|
\]
with $\iota=\frac{\gamma(1+{c})}{2(1-{c})^2}+\frac{2}{1-{c}}$.
Due to the arbitrariness of $\epsilon\in(0,1)$, the estimation in (\ref{eq6-2}) holds for $k=K+1$
by reducing $\epsilon$ if necessary so that
$\iota\epsilon\leq\hat{c}-\bar{c}$. This completes the proof.
\end{proof}

The important technique in the proof of Theorem \ref{theorem2} is the decomposition method of $F(u_{k+1})$, especially the structure and analysis of $N_k$,
which reduces the Lipschitz constant of the nonosmooth part of $B_k$ by half.

In EDIIS(1), $\alpha_k$ is chosen as the minimizer of the optimization problem
\[\min \frac{1}{2}\left\|(1-\alpha)F_k+\alpha F_{k-1}\right\|^2, \quad {\rm s.t.} \quad 0\le  \alpha\le 1. \]
This is a convex optimization problem
and its solution $\alpha_k$
can be expressed by the formulation with the middle operator as
$$
\alpha_k={\rm mid} \left\{ 0, \, \frac{F_k^{\rm T}(F_{k}-F_{k-1})}{\|F_{k-1}-F_k\|^2}, \, 1\right\}
\footnote{mid$(0,a,1)=\left\{\begin{array}{ll}
0, & a<0\\
a, & a\in [0,1]\\
1, & a >1.
\end{array}
\right.
$}.
$$
Following the proof of Theorem \ref{theorem2}, it is clear that (\ref{eq26-2}) and (\ref{eq23}) hold for
$\alpha_k=0$ and $\alpha_k=1$, which are the points that we only need to check for the EDIIS(1) with respect to Anderson(1). Thus, we have the following statement.
\begin{corollary}\label{corollary2}
Suppose that the assumptions of Theorem \ref{theorem2} hold. Then the sequence $\{u_k\}$ generated by
EDIIS(1) satisfies (\ref{eq6-2}).
\end{corollary}

Since the results in Theorem \ref{theorem2} and Corollary \ref{corollary2} are local convergence results of Anderson(1) and EDIIS(1), the Lipschitz continuous differentiability of $Q$ and $H$ around $u^*$ and $P_{\Omega}(Q(u^*))$ is enough to guarantee these statements.
\section{Anderson acceleration method with smoothing approximation}\label{section3}
\subsection{Smoothing approximation}\label{section3.1}
In this subsection, we introduce some smoothing approximations of the nonsmooth contraction mapping $G$ for
finding its fixed point. For a function $\omega:\mathbb{R}^n\times(0,1]\rightarrow\mathbb{R}^n$,
$\omega'(y,\mu)$ always denotes the derivative of $\omega$ with respect to $y$ for fixed $\mu\in(0,1]$ in what follows. We define a smoothing function of $\max\{t,0\}$ at first.
\begin{definition}\label{smoothing}
\cite{Chen} We call $\psi:\mathbb{R}\times(0,1]\rightarrow\mathbb{R}$ a smoothing function of $\max\{t,0\}$ in $\mathbb{R}$, if $\psi(\cdot,\mu)$ is continuously differentiable in
$\mathbb{R}$ for any fixed $\mu>0$, and the following conditions hold.
\begin{itemize}
\item[\rm(i)] There is a $\kappa_{\psi}>0$ such that
for any $t\in \mathbb{R}$ and $\mu\in(0,1]$, $|\psi(t,\mu)-\max\{t,0\}|\le \kappa_{\psi}\mu$.
\item [\rm(ii)] For any $t\in \mathbb{R}$, it holds
$\left\{\lim_{s\rightarrow t\,\mu\downarrow0}\psi'(s,\mu)\right\}\subseteq\partial \left(\max\{t,0\}\right)$,
where $\partial$ indicates the Clarke subdifferential \cite{Clarke}.
\end{itemize}
\end{definition}

Definition \ref{smoothing}-(i) implies that
$
\lim_{s\rightarrow t\,\mu\downarrow0}\psi(s,\mu)=\max\{t,0\}$ and Definition \ref{smoothing}-(ii) implies the gradient consistency.
Smoothing functions for the $\max$ function have been studied in numerical methods for optimization and differential equations \cite{Chen}.
Four widely used smoothing functions of $\max\{t,0\}$ are as follows:
\begin{equation}\label{eq1-2}
\begin{aligned}
&\psi_1(t,\mu)=t+\mu\ln(1+e^{-\frac{t}{\mu}}),
\quad
&&\psi_2(t,\mu)=\frac{1}{2}(t+\sqrt{t^2+4\mu^2}),\\
&\psi_3(t,\mu)=\left\{\begin{aligned}
&\max\{0,t\} &&\mbox{if}\, |t|>\mu\\
&\frac{(t+\mu)^2}{4\mu} &&\mbox{if}\, |t|\leq \mu,
\end{aligned}\right.\quad &&\psi_4(t,\mu)=\left\{\begin{aligned}
&t +\frac{\mu}{2}e^{-\frac{t}{\mu}} &&\mbox{if}\, t> 0\\
&\frac{\mu}{2}e^{\frac{t}{\mu}} &&\mbox{if}\, t \leq 0.
\end{aligned}\right.
\end{aligned}
\end{equation}
Let $\psi$ be a smoothing function of $\max\{t,0\}$. For $v\in\mathbb{R}^l$, set
\begin{equation}\label{eq13}
\Phi(v,
\mu)=(\phi_1(v_1,\mu),\phi_2(v_2,\mu),\ldots,\phi_l(v_l,\mu))^T,
\end{equation}
where
\begin{equation}\label{eq20}
\phi_i(t,\mu)=\psi(\underline{w}_i-t,\mu)+t-\psi(t-\overline{w}_i,\mu), \,\,i=1,2,\ldots,l.
\end{equation}
It is clear that $\Phi(\cdot,\mu)$ is continuously differentiable on $\mathbb{R}^l$ for any fixed $\mu\in(0,1]$, and by \eqref{eq-P1}, we have
\begin{equation}\label{eq12-2}
\lim_{s\rightarrow t\,\mu\downarrow0}\phi_i(s,\mu)=P_{[\underline{w}_i,\overline{w}_i]}(t)\quad\mbox{and}\quad
|\phi_i(t,\mu)-P_{[\underline{w}_i,\overline{w}_i]}(t)|\leq 2\kappa_{\psi}\mu, \quad\forall t\in\mathbb{R},\,\mu\in(0,1].
\end{equation}
Then, since $\underline{w}_i<\overline{w}_i$ for all $i=1,2,\ldots,l$, we obtain
\begin{equation}\label{eq12}
\left\{\begin{aligned}
&\lim_{s\rightarrow{t},\mu\downarrow0}\phi_i'(s,\mu)=\{0\}&&\mbox{if ${t}<\underline{w}_i$ or ${t}>\overline{w}_i$}\\
&\lim_{s\rightarrow{t},\mu\downarrow0}\phi_i'(s,\mu)=\{1\}&&\mbox{if $\underline{w}_i<{t}<\overline{w}_i$}\\
&\left\{\lim_{s\rightarrow{t},\mu\downarrow0}\phi_i'(s,\mu)\right\}\subseteq
[0,1]&&\mbox{if ${t}=\underline{w}_i$ or ${t}=\overline{w}_i$}.
\end{aligned}\right.
\end{equation}
\begin{proposition}\label{p4_3}
Let $\psi$ be a smoothing function of $\max\{t,0\}$ with parameter $\kappa_{\psi}$ in Definition \ref{smoothing}-(i).
Suppose Assumption \ref{ass-Q-H} holds and $\Omega+\mathcal{B}(2\kappa_{\psi}\sqrt{l},\textbf{0})\subseteq D_H$, then
the function
\begin{equation}\label{G-s}
\mathcal{G}(u,\mu)=H(\Phi(Q(u),\mu))
\end{equation}
owns the following properties.
\begin{itemize}
\item[\rm(i)] $\mathcal{G}(\cdot,\mu)$ is continuously differentiable on $D$ for any fixed $\mu\in(0,1]$.
\item[\rm(ii)] There is a $\kappa_{G}>0$ such that
for any $u\in D$ and $\mu\in(0,1]$, $\|\mathcal{{G}}(u,\mu)-G(u)\|\le \kappa_G \mu$.
\item[\rm(iii)] For any $u\in D$,
$\limsup_{z\rightarrow u,\mu\downarrow0}\|\mathcal{G}'(z,\mu)\|\leq c$.
\item [\rm(iv)] For any ${c}_S\in({c},1)$, there exists $\hat{\mu}\in(0,1]$ such that
for any fixed $\mu\in(0,\hat{\mu}]$, $\|\mathcal{{G}}'(u,\mu)\|\leq{c}_S$,
$\forall u\in D$, which implies that $\mathcal{{G}}(\cdot,\mu)$ is a contraction
mapping on $D$ with factor $c_S$, i.e.
\begin{equation}\label{lemma-G-1}
\|\mathcal{{G}}(u,\mu)-\mathcal{{G}}(v,\mu)\|\leq {c}_S\|u-v\|,\quad \mbox{for all $u,v\in D$, $\mu\in(0,\hat{\mu}]$}.
\end{equation}
\item [\rm(v)] Let $u_{\mu}$ be a fixed point of $\mathcal{{G}}(\cdot,\mu)$, then $\|u_{\mu}-u^*\|\leq\left(\frac{\kappa_G}{1-c}\right)\mu$,
which further implies $\lim_{\mu\rightarrow0}u_{\mu}=u^*$.
\end{itemize}
\end{proposition}
\begin{proof}
From \eqref{eq12-2},
we can claim that
\begin{equation}\label{eq15}
\|\Phi(Q(u),\mu)-P_{\Omega}(Q(u))\|\leq 2\kappa_{\psi}\sqrt{l}\mu.
\end{equation}
Since $\Omega+\mathcal{B}(2\kappa_{\psi}\sqrt{l},\textbf{0})\subseteq D_H$,
by the continuous differentiability of $Q$, $H$ and $\Phi(\cdot,\mu)$, (i) and (ii) hold with $\kappa_G=2{c}_H\kappa_{\psi}\sqrt{l}$.

Note that
\begin{equation}\label{eq16}
\mathcal{G}'(z,\mu)=H'(w)_{w=\Phi(Q(z),\mu)}\Phi'(v,\mu)_{v=Q(z)}Q'(z).
\end{equation}
Recalling \eqref{eq12} and the definition of $\Phi$ in \eqref{eq13}, we get
\begin{equation}\label{eq17-2}
\|\Phi'(v,\mu)_{v=Q(z)}\|
=\|{\rm diag}(\phi'_i(v_i,\mu)_{v_i=Q_i(z)})\|\leq1.
\end{equation}
Then, the continuous differentiability of $H$ and $Q$ combining with the estimations in \eqref{eq15}, \eqref{eq16} and \eqref{eq17-2} gives that
$$\limsup_{z\rightarrow u,\mu\downarrow0}\|\mathcal{G}'(z,\mu)\|\leq
\|H'(w)_{w=P_{\Omega}(Q(u))}\|\|Q'(u)\|\leq c_Hc_Q=c,$$
which guarantees items (iii) and (iv).

Since $u_{\mu}$ and $u^*$ are the fixed points of $\mathcal{{G}}(\cdot,\mu)$ and $G$ on $D$, respectively, by (ii), we have
$$
\|u_{\mu}-u^*\|=
\|\mathcal{{G}}(u_{\mu},\mu)-G(u^*)\|
\leq
\|\mathcal{{G}}(u_{\mu},\mu)-G(u_{\mu})\|+\|G(u_{\mu})-G(u^*)\|
\leq \kappa_G\mu+c\|u_{\mu}-u^*\|,
$$
which gives the results in (v) by simple deduction.
We complete the proof.
\end{proof}

If $G$ satisfies Assumption \ref{ass-Q-H}, Proposition \ref{p4_3}-(iv) says that its smoothing approximations in \eqref{G-s} also own the contractive property when $\mu$ is sufficiently small.
Inspired by the proof of Proposition \ref{p4_3}, if $u^k$ is an approximate fixed point of $\mathcal{{G}}(u,\mu_k)$ with accuracy tolerance $\epsilon_k$, i.e.
$\|\mathcal{{G}}(u^k,\mu_k)-u^k\|\leq\epsilon_k,$
then we also have $\lim_{k\rightarrow\infty}u^k=u^*$, if $\lim_{k\rightarrow\infty}\mu_k=0$ and
$\lim_{k\rightarrow\infty}\epsilon_k=0$. Moreover, the error estimation in Proposition \ref{p4_3}-(v) holds
always no matter $\mathcal{G}(\cdot,\mu)$ is contractive or not.
Proposition \ref{p4_3}-(v) also gives an upper bound of
the error on the fixed point of $G$ and its smoothing approximation, which is defined by the parameter $\kappa_G$
coming from the structure of the smoothing approximation function and the contraction factor of $G$.
\begin{remark}\label{remark2}
Following the proof of Proposition \ref{p4_3}, condition $\Omega+\mathcal{B}(2\kappa_{\psi}\sqrt{l},\textbf{0})\subseteq D_H$ is only used to guarantee
$\Phi(Q(u),\mu)\in D_H$ for all $u\in D$ and $\mu\in (0,1]$. So, the statements (i) and (ii) in Proposition \ref{p4_3} hold for any $\mu\in(0,\tilde{\mu}]$ with parameter $\tilde{\mu}\in(0,1]$
satisfying $\Omega+\mathcal{B}(2\kappa_{\psi}\sqrt{l}\tilde{\mu},\textbf{0})\subseteq D_H$.
\end{remark}
\subsection{A modified Anderson(m) algorithm}
In this subsection, we will propose an Anderson acceleration algorithm for the nonsmooth fixed point problem \eqref{eq1-0} based on the smoothing approximation method. At first,
we study the new smoothing function of $\max\{t,0\}$ as follows, which has more desirable properties for solving \eqref{eq1-0}:
\begin{equation}\label{smoothing1}
\psi(t,\mu)=\left\{
\begin{aligned}
&0&&\mbox{if $t\leq0$}\\
&\frac{t^2}{2\mu}&&\mbox{if $0< t\leq\mu$}\\
&\frac{1}{4}(t-\mu)^2+t-\frac{1}{2}\mu&&\mbox{if $\mu<t\leq\mu+\sqrt{\mu}$}\\
&-\frac{1}{4}(t-\mu-2\sqrt{\mu})^2+t&&\mbox{if $\mu+\sqrt{\mu}< t\leq\mu+2\sqrt{\mu}$}\\
&t&&\mbox{if $t>\mu+2\sqrt{\mu}$.}\\
\end{aligned}
\right.
\end{equation}
Fig. \ref{fig-h} shows the smoothing function $\psi(\cdot,\mu)$ in \eqref{smoothing1} with different values of $\mu$, while Fig. \ref{fig-h2} shows the relationships of $\max\{t,0\}$ and its smoothing functions
defined in \eqref{eq1-2} and  \eqref{smoothing1}.
 Since $\psi$ in \eqref{smoothing1} is a smoothing function of $\max\{t,0\}$ with Definition \ref{smoothing}, the results in Proposition \ref{p4_3} also holds for $\mathcal{G}(u,\mu)$ defined in \eqref{G-s} with $\psi$ in \eqref{smoothing1}. In what follows, we will present some more desirable properties of $\psi$ in \eqref{smoothing1}.
\begin{figure}
\begin{center}
  \subfigure[]{
   \label{fig-h}
    \includegraphics[width=2.5in,height=2.2in]{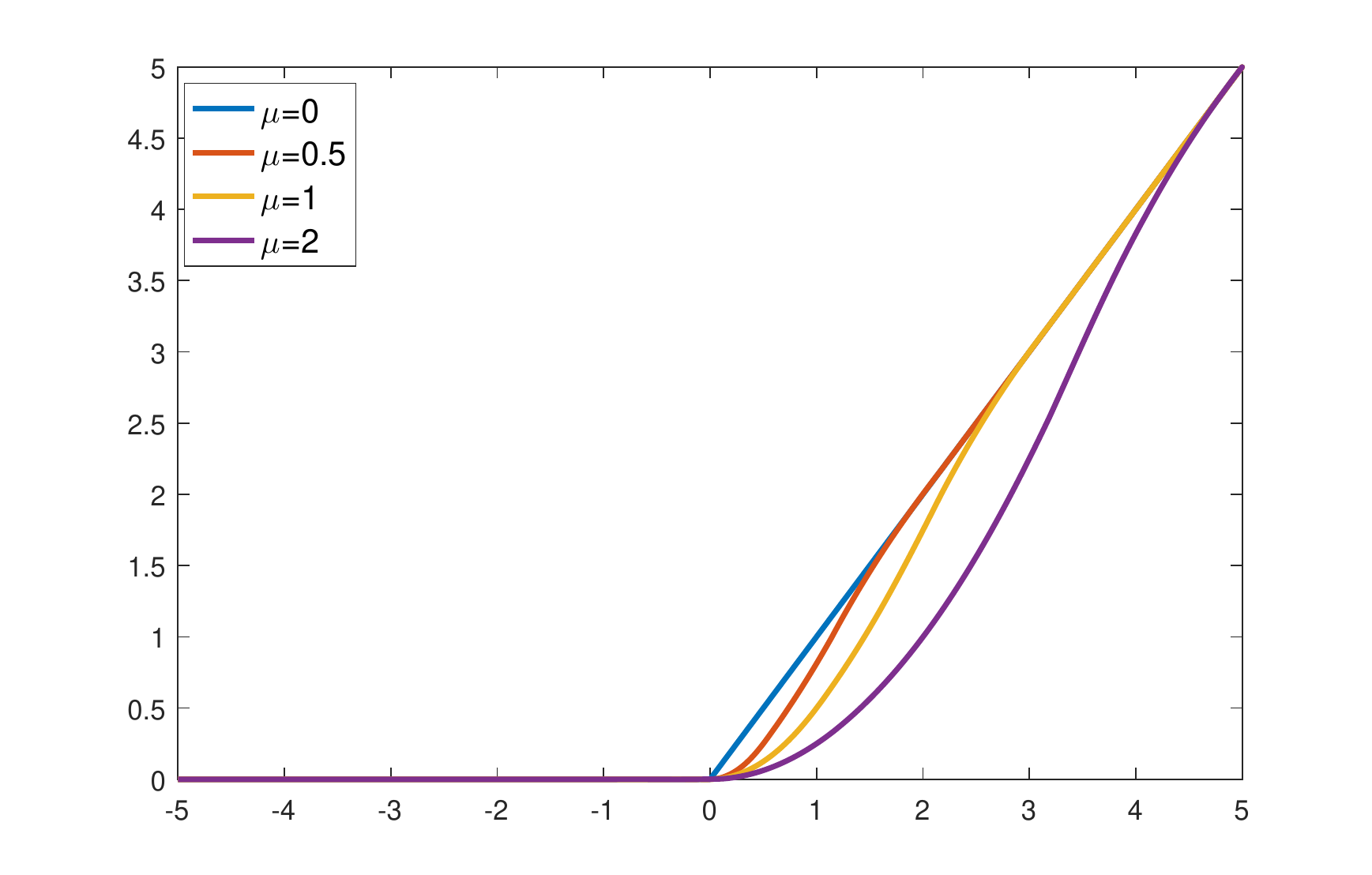}}
  \subfigure[ ]{
    \label{fig-h2}
    \includegraphics[width=2.5in,height=2.2in]{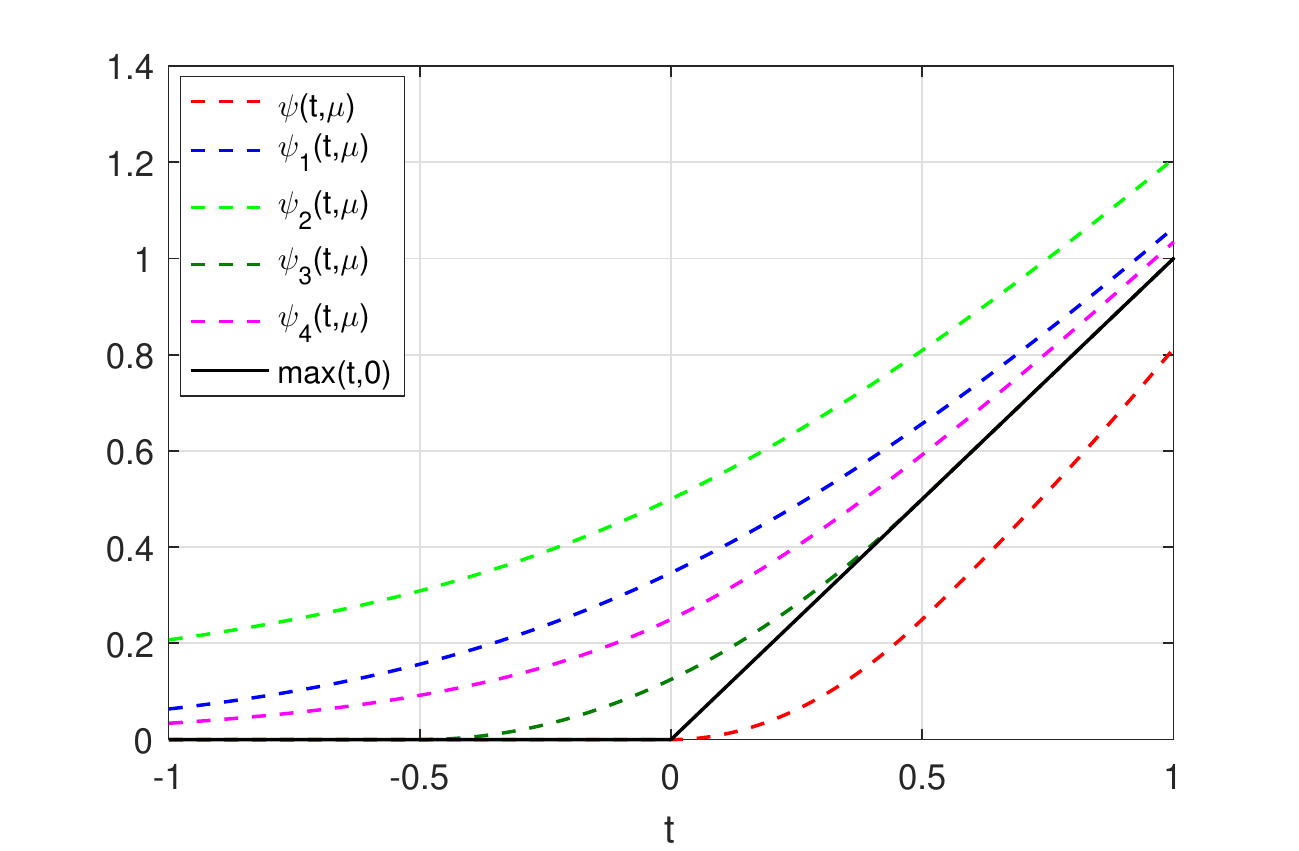}}
\caption{Smoothing functions of $\max\{t,0\}$: (a) $\psi(t,\mu)$ in \eqref{smoothing1} with different values of $\mu$; (b) $\max\{t,0\}$, $\psi(t,\mu)$ in \eqref{smoothing1} and the four smoothing functions $\psi_i(t,\mu)$ in \eqref{eq1-2} with $\mu=0.3$.}
\end{center}
\end{figure}
\begin{proposition}\label{prop2}
Function $\psi(t,\mu)$ in (\ref{smoothing1}) is continuously differentiable with respect to $t$ for any fixed $\mu\in(0,1]$ and satisfies the following properties.
\begin{itemize}
\item [{\rm (i)}] $|\psi(t,\mu)-\max\{t,0\}|\leq\frac{1}{2}\mu$, for any $t\in\mathbb{R}$ and $\mu\in(0,1]$.
\item [{\rm (ii)}]
$\psi(t,\mu)=\max\{t,0\}$ if $t\leq0$ or $t\geq\mu+2\sqrt{\mu}$.
\item [{\rm (iii)}] For any $\mu\in(0,1]$,
$\psi'(t,\mu)=0$ if $t\leq0$, $0\leq\psi'(t,\mu)\leq 1+\frac{1}{2}\sqrt{\mu}$ if $0<t<\mu+2\sqrt{\mu}$,
and
$\psi'(t,\mu)=1$ if $t\geq\mu+2\sqrt{\mu}$.
\end{itemize}
\end{proposition}
\begin{proof}
By the definition of $\psi$ in (\ref{smoothing1}), we obtain
$$\left|\psi(t,\mu)-\max\{t,0\}\right|=\left\{
\begin{aligned}
&0&&\mbox{if $t\leq0$}\\
&|{t^2}/{2\mu}-t|\leq\mu/2&&\mbox{if $0<t\leq\mu$}\\
&|(t-\mu)^2/4-\mu/2|\leq\mu/2&&\mbox{if $\mu<t\leq\mu+\sqrt{\mu}$}\\
&(t-\mu-2\sqrt{\mu})^2/4\leq\mu/4&&\mbox{if $\mu+\sqrt{\mu}< t\leq\mu+2\sqrt{\mu}$}\\
&0&&\mbox{if $t>\mu+2\sqrt{\mu}$,}\\
\end{aligned}
\right.$$
which implies the statements in (i) and (ii).

By straightforward calculation, we can verify
 that $\psi(t,\mu)$ is continuously differentiable with respect to $t$ for any fixed $\mu\in(0,1]$ and the estimation in (iii) holds.
\end{proof}

By Proposition \ref{prop2}-(ii), it holds that for any fixed $t\in\mathbb{R}$, there exists $\bar{\mu}>0$ such that $\psi(t,\mu)=\max\{t,0\}$, $\forall \mu\in(0,\bar{\mu}]$, which is
the main advantage of $\psi$ in \eqref{smoothing1} compared with the other four smoothing functions of $\max\{t,0\}$ in \eqref{eq1-2}. Following the proof of Proposition \ref{prop2}, we can further obtain the following
properties of $\phi_i$ in \eqref{eq20} with $\psi$ in \eqref{smoothing1}.
\begin{proposition}\label{prop3}
For any fixed $\mu\in(0,1]$, functions $\phi_i(\cdot,\mu)$ in \eqref{eq20} with $\psi$ in \eqref{smoothing1}, $i=1,2,\ldots,l$, are continuously differentiable and satisfy the following properties:
\begin{itemize}
\item [{\rm (i)}] $|\phi_i(t,\mu)-P_{[\underline{w}_i,\overline{w}_i]}(t)|\leq\frac{1}{2}\mu$,
for any $t\in\mathbb{R}$;
\item [{\rm (ii)}]
$\phi_i(t,\mu)=P_{[\underline{w}_i,\overline{w}_i]}(t)$ if $t\leq \underline{w}_i-\mu-2\sqrt{\mu}$ or
$\underline{w}_i\leq t\leq\overline{w}_i$ or $t\geq\overline{w}_i+\mu+2\sqrt{\mu}$;
\item [{\rm (iii)}] $|\phi_i'(t,\mu)|\leq 1$, for any $t\in\mathbb{R}$.
\end{itemize}
\end{proposition}
\begin{proof}
By Proposition \ref{prop2}-(ii), we have
$$\begin{aligned}
&\psi(\underline{w}_i-t,\mu)=\max\{\underline{w}_i-t,0\}&&\mbox{if $t\geq\underline{w}_i$ or
$t\leq\underline{w}_i-\mu-2\sqrt{\mu}$},\\
&\psi(t-\overline{w}_i,\mu)=\max\{t-\overline{w}_i,0\}&&\mbox{if $t\leq\overline{w}_i$ or
$t\geq\overline{w}_i+\mu+2\sqrt{\mu}$}.
\end{aligned}$$
Then, for any $\mu\in(0,1]$ and $t\in \mathbb{R}$,
at most one of $\psi(\underline{w}_i-t,\mu)=\max\{\underline{w}_i-t,0\}$ and $\psi(t-\overline{w}_i,\mu)=\max\{t-\overline{w}_i,0\}$ holds. Then,
the results (i) and (ii) in Proposition \ref{prop2} imply items (i) and (ii) in this proposition.

In what follows, we consider the estimation in item (iii). From \eqref{smoothing1}, we have
\begin{equation}\label{eq5}
\begin{aligned}
|\phi_i'(t,\mu)|=&|\psi'(\underline{w}_i-t,\mu)+1-\psi'(t-\overline{w}_i,\mu)|\\
=&\left\{
\begin{aligned}
&0&&\mbox{if $t\leq\underline{w}_i-\mu-2\sqrt{\mu}$}\\
&|\underline{w}_i-t-\mu-2\sqrt{\mu}|/2\leq \sqrt{\mu}/2&&
\mbox{if $\underline{w}_i-\mu-2\sqrt{\mu}\leq t<\underline{w}_i-\mu-\sqrt{\mu}$}\\
&|t+\mu-\underline{w}_i|/2\leq\sqrt{\mu}/2&&
\mbox{if $\underline{w}_i-\mu-\sqrt{\mu}<t<\underline{w}_i-\mu$}\\
&|t-\underline{w}_i+\mu|/{\mu}\leq1&&\mbox{if $\underline{w}_i-\mu\leq t<\underline{w}_i$}\\
&1&&\mbox{if $\underline{w}_i\leq t<\overline{w}_i$}\\
&|\mu+\overline{w}_i-t|/{\mu}\leq1&&\mbox{if $\overline{w}_i\leq t<\overline{w}_i+\mu$}\\
&|t-\overline{w}_i-\mu|/2\leq \sqrt{\mu}/2&&\mbox{if $\overline{w}_i+\mu\leq t<\overline{w}_i+\mu+\sqrt{\mu}$}\\
&|t-\overline{w}_i-\mu-2\sqrt{\mu}|/2\leq \sqrt{\mu}/2&&\mbox{if $\overline{w}_i+\mu+\sqrt{\mu}\leq t<\overline{w}_i+\mu+2\sqrt{\mu}$}\\
&0&&\mbox{if $t\geq\overline{w}_i+\mu+2\sqrt{\mu}$.}\\
\end{aligned}\right.\end{aligned}
\end{equation}
Thus, (iii) holds.
\end{proof}

In what follows, we will use the smoothing function of $\max\{t,0\}$ in \eqref{smoothing1} to construct a smoothing approximation of $P_{\Omega}(v)$ on $\mathbb{R}^l$, which is also with the formulation in \eqref{eq13}. Then, we can give a
smoothing approximation of $G$ in \eqref{eq1-0} by the formulation of \eqref{G-s} with \eqref{smoothing1}.

Set $\varpi_1=\min\{3,\underline{w}_i-Q_i(u^*): i\in\{i:Q_i(u^*)<\underline{w}_i\}\}$, $\varpi_2=\min\{3,Q_i(u^*)-\overline{w}_i:i\in\{i:Q_i(u^*)>\overline{w}_i\}\}$, and
by Assumption \ref{ass-Q-H}-(ii), denote $\eta\in(0,1]$ the parameter such that
\begin{equation}\label{eq-eta}
 \Omega+\mathcal{B}({\sqrt{l}\eta}/{2},\textbf{0})\subseteq D_H.
\end{equation}
Then, we define parameter $\bar{\mu}$ by
\begin{equation}\label{eq_mu}
\bar{\mu}=\min\{\eta,(\varpi_1/3)^2,(\varpi_2/3)^2\}.
\end{equation}

\begin{theorem}\label{thm1} Suppose Assumption \ref{ass-Q-H} holds. Besides the properties in Proposition \ref{p4_3},
function $\mathcal{G}(u,\mu)$ in \eqref{G-s} with $\psi$ defined in \eqref{smoothing1} owns the following properties.
\begin{itemize}
\item [{\rm (i)}] For any fixed $\mu\in(0,\eta]$, $\mathcal{G}(\cdot,\mu)$ is a contractive mapping on $D$
    with contraction factor no larger than ${c}$ in Assumption \ref{ass-Q-H}.
\item [{\rm (ii)}]
$\|\mathcal{G}(u,\mu)-G(u)\|\leq\kappa\mu$ for all $u\in D$ and $\mu\in(0,\eta]$ with $\kappa=c_H{\sqrt{l}}/{2}$;
\item [{\rm (iii)}] $\mathcal{G}(u^*,\mu)=G(u^*)=u^*$, $\forall \mu\in(0,\bar{\mu}]$, where $\bar{\mu}$ is defined by \eqref{eq_mu}.
\end{itemize}
\end{theorem}
\begin{proof}
By Proposition \ref{prop3}-(i), it holds
\begin{equation}\label{eq18}
\|\Phi(Q(u),\mu)-P_{\Omega}(Q(u))\|\leq{\sqrt{l}}\mu/{2}.
\end{equation}
Then, $\Phi(Q(u),\mu)\in D_H$ for all $u\in D$ and $\mu\in(0,\eta]$ can be guaranteed by
the condition $\Omega+\mathcal{B}({\sqrt{l}}\eta/{2},\textbf{0})\subseteq D_H$.

(i)
Using the Lipschitz property of $H$ and $Q$ again, for any $u,v\in D$ and $\mu\in(0,\eta]$, we obtain
$$\begin{aligned}
\|\mathcal{G}(u,\mu)-\mathcal{G}(v,\mu)\|\leq&c_H\|\Phi(Q(u),\mu)-\Phi(Q(v),\mu)\|\\
\leq&c_H\|Q(u)-Q(v)\|
\leq c_Hc_Q\|u-v\|
={c}\|u-v\|,
\end{aligned}$$
where the second inequality follows from Proposition \ref{prop3}-(iii).
Thus, for any $\mu\in(0,\eta]$, $\mathcal{G}(u,\mu)$ is a contractive mapping on $D$ with factor no larger than ${c}$.

(ii) By the Lipschitz property of $H$ on $D_H$ and $\Phi(Q(u),\mu)\in D_H$ for all $u\in D$ and $\mu\in(0,\eta]$, it holds
\[\begin{aligned}
\|\mathcal{G}(u,\mu)-G(u)\|=&\|H(\Phi(Q(u),\mu))-H(\max\{Q(u),0\})\|\\
\leq& c_H\|\Phi(Q(u),\mu)-\max\{Q(u),0\}\|\leq \kappa\mu,
\end{aligned}\]
where the last inequality follows from \eqref{eq18} with $\kappa={{\sqrt{l}}c_H/{2}}$.

(iii) Denote $I_1=\{i:Q_i(u^*)<\underline{w}_i\}$, $I_2=\{i:\underline{w}_i\leq Q_i(u^*)\leq
\overline{w}_i\}$ and $I_3=\{i:Q_i(u^*)>\overline{w}_i\}$. First, we can easily find that
$$\phi_i(Q_i(u^*),\mu)=Q_i(u^*)=P_{[\underline{w}_i,\overline{w}_i]}(Q_i(u^*)),\quad \forall i\in I_2.$$
Next, for $i\in I_1$, by the definition of $\varpi_1$ and $\bar{\mu}\leq(\varpi_1/3)^2\leq1$, we have
$$Q_i(u^*)\leq\underline{w}_i-\varpi_1=\underline{w}_i-3\sqrt{\bar{\mu}}\leq\underline{w}_i-
\mu-2\sqrt{\mu},\quad \forall i\in I_1,\,0<\mu\leq\bar{\mu}, $$
by Proposition \ref{prop3}-(ii), which implies
\begin{equation}\label{eq25}
\phi_i(Q_i(u^*),\mu)=P_{[\underline{w}_i,\overline{w}_i]}(Q_i(u^*)),\quad \forall i\in I_1,\,0<\mu\leq\bar{\mu}.
\end{equation}
Similarly, for $i\in I_3$, we obtain
$$Q_i(u^*)\geq\overline{w}_i+\mu+2\sqrt{\mu},\quad \forall i\in I_3,\,0<\mu\leq\bar{\mu}, $$
which gives \eqref{eq25} for $i\in I_3$. Thus, for any $\mu\in(0,\bar{\mu}]$, we have $\Phi(Q(u^*),\mu)=P_{\Omega}(Q(u^*))$ and thus
$\mathcal{G}(u^*,\mu)=G(u^*)=u^*.$
We complete the proof.
\end{proof}

Inspired by Theorem \ref{thm1}-(iii), when $\mu\leq\bar{\mu}$ with $\bar{\mu}$ defined in \eqref{eq_mu}, $u^*$ is also the fixed point of $\mathcal{{G}}(u,\mu)$, and from Theorem \ref{thm1}-(i), we further have
\begin{equation}\label{smoothing2}
(1-c)\|u-u^*\|\leq \|\mathcal{F}(u,\mu)\|\leq (1+c)\|u-u^*\|,\quad \forall u\in D,\,\mu\in(0,\bar{\mu}],
\end{equation}
where $\mathcal{F}(u,\mu)=\mathcal{G}(u,\mu)-u$.
\begin{remark}\label{remark4}
 Proposition \ref{prop3}-(ii) shows that ${\cal G}(u, \mu)=G(u)$, for any $\mu \in (0,1]$ and $u\in D$ satisfying $Q(u)\in\Omega$. Thus,
if $u^*$ is the fixed point of ${\cal G}(\cdot, \mu)$ for a given $\mu\in(0,1]$ and
$Q(u^*)\in\Omega$, then we can justify that $u^*$ is also the fixed point of $G$.
\end{remark}
\begin{algorithm}[ht]
\caption{s-Anderson(m)}\label{algo-o1}
\textbf{Choose} $u_0\in D$ and a positive integer $m$.

\textbf{Set} parameters $\sigma_1,\sigma_2\in(0,1)$, $\gamma>0$ and a sufficiently small positive parameter $\epsilon<\gamma\|F(u_0)\|^2$.

\textbf{Let}
$F_0=G(u_0)-u_0$, $\mu_0=\gamma\|F_0\|^2$,
${\cal F}_0={\cal G}(u_0,\mu_0)-u_0$ and $u_1={\cal G}(u_0,\mu_0)$.

\textbf{for $k=1,2,...$ do}

\qquad  set $F_k=G(u_k)-u_k$,

\qquad if $\|F_k\|\leq\sigma_1\|F_{k-1}\|$, then let
$$\mu_{k}=\mu_{k-1},$$
\qquad otherwise, let
$$\mu_{k}=\max\{\epsilon,\sigma_2\mu_{k-1}\};$$

\qquad  set ${\cal F}_k={\cal G}(u_k,\mu_k)-u_k$;

\qquad choose $m_k=\min\{m,k\}$;

\qquad solve
\begin{equation}\label{eq7}
\min\,\,\left\|\sum_{j=0}^{m_k}\alpha_j{{\cal F}}_{k-m_k+j}\right\|\quad \quad
\mbox{s.t.}\,\sum_{j=0}^{m_k}\alpha_j=1
\end{equation}
\qquad to find a solution $\{\alpha_j^k:j=0,\ldots,m_k\}$, and set
\begin{equation}\label{uk}
u_{k+1}=\sum_{j=0}^{m_k}\alpha_j^k{{\cal G}}(u_{k-m_k+j},\mu_{k-m_k+j});
\end{equation}
\textbf{end for}
\end{algorithm}

By Theorem \ref{thm1}, when we use \eqref{G-s}
with (\ref{smoothing1}) as the smoothing approximation of $G$,
$\mathcal{G}(\cdot,\mu)$ is contractive
and $u_{\mu}=u^*$ for $\mu\in(0,\bar{\mu}]$, where $u_{\mu}$ is the fixed point of $\mathcal{G}(\cdot,\mu)$. Then, we can apply Anderson(m) or EDIIS(m) to find a fixed point of $G$ by using ${\cal G}(\cdot, \mu)$ in the algorithms.
{If $u_0$ is sufficiently close to $u^*$, then $\mu_0:=\gamma\|F(u_0)\|^2<\bar\mu$. In such case,  we can let $\mu_k:=\mu_0$ for all $k$. However, $u^*$ is unknown, and the value of $\bar{\mu}$ in (3.14) is often difficult to be evaluated in practice. Thus, we use an updating scheme on $\mu_k$ in Algorithm 3.1 to improve the ability and performance of the Anderson acceleration methods for nonsmooth fixed point problems.}
In s-Anderson(m), we replace $G(u)$ in Anderson(m) by ${\cal G}(u,\mu)$ and update $\mu$ step by step. {The strategy for updating $\mu_k$ in Algorithm 3.1 is based on the reduction of the norms of the residual function at $u_k$ and $u_{k-1}$.
If $\|F_k\|\leq\sigma_1\|F_{k-1}\|$, then it means that using $\mu_{k-1}$ can reduce the norm of the residual function at $u_k$ sufficiently. Hence we let $\mu_k=\mu_{k-1}$ for the next iteration. Otherwise, we
set $\mu_k = \max\{\epsilon, \sigma_2 \mu_{k-1}\}.$}

Same as the condition on the coefficients $\{\alpha_j^k:j=1,\ldots,m_k\}$
used in \cite{ChenKelley2015,TothKelley2015}, we need the following assumption on them in \eqref{eq7}.
\begin{assumption}\label{ass2-2}
There exists an $M_{\alpha}\geq 1$ such that $\sum_{j=0}^{m_k}|\alpha_j^k|\leq M_{\alpha}$ holds for all $k\geq1$.
\end{assumption}

Before proving the local r-linear convergence of s-Anderson(m), we need predefine some necessary parameters used in the forthcoming proof and give some preliminary analysis.
\begin{itemize}
\item $a$:
Combining \eqref{eq16}, \eqref{eq17-2} with the Lipschitz property of $Q'(u)$, ${\rm diag}(\phi'(Q_i(u),\mu))$ and $H'(\Phi(Q(u),\mu))$ on $D$, there exists a constant $a>0$ such that $\mathcal{G}'(u,\mu)$ is Lipschitz continuous on $\mathcal{B}({\delta},u^*)$ with constant $a$. This means,
\begin{equation}\label{eq-G}
{\mathcal{G}}(u,\mu)={\mathcal{G}}(u^*,\mu)+{\mathcal{G}}'(u^*,\mu)(u-u^*)+\Delta_{u},\quad \forall u\in\mathcal{B}({\delta},u^*),\,\mu\in[\epsilon,\eta],
\end{equation}
where $\|\Delta_{u}\|\leq\frac{1}{2}a\|u-u^*\|^2$.
\item $\delta_{1}$: Since $\mathcal{G}(u,\mu)$ is Lipschitz continuous, from Theorem 2.2 in \cite{ChenKelley2015}, there exists $\delta_{1}\in(0,\delta]$ such that if $\|u_0-u^*\|\leq{\delta_{1}}$, we have the r-linear convergence of Anderson(m) on solving $\mathcal{F}(u,\hat{\mu}):=\mathcal{G}(u,\hat{\mu})-u=0$ with any $\hat{\mu}\in[\epsilon,\bar{\mu}]$, that is
\begin{equation}\label{eq_mu2}
\limsup_{k\rightarrow\infty}\left(\frac{\|\mathcal{F}(u_k,\hat{\mu})\|}{\|\mathcal{F}(u_0,\hat{\mu})\|}
\right)^{{1}/{k}}\leq {c},
\end{equation}
where ${c}$ is a contraction factor of $\mathcal{G}(u,\hat{\mu})$ on $D$ by Theorem \ref{thm1}-(i).
\item $\delta_0$: Let
\begin{equation}\label{eq-u0}
\delta_0:=\min\{\delta_1, \frac{\sqrt{\bar{\mu}}}{\sqrt{\gamma}(1+c)},\frac{\sqrt{\eta}}{\sqrt{\gamma}(1+c)},
\frac{(1-c)\delta_1}{M_{\alpha}(1+c)},
\frac{1-c}{\varpi}\},
\end{equation}
where $\bar{\mu}$ is defined in \eqref{eq_mu} and $\varpi=\frac{a(M_{\alpha}^2+M_{\alpha})(1+c)+2M_{\alpha}\sqrt{l}c_H\gamma(1+c)^2(1-c)}{2(1-c)^2}$.
\end{itemize}

\begin{lemma}\label{lemma1}
If $\|u_0-u^*\|\leq\delta_0$, then {for the sequences $\{\mu_k\}$, $\{u^k\}$ and $\{\mathcal{F}_k\}$ generated by s-Anderson(m) in Algorithm \ref{algo-o1}}, it holds that
\begin{eqnarray}
\mu_k\leq\bar{\mu},\quad\Omega+\mathcal{B}(\sqrt{l}\mu_k/2,\textbf{0})\subseteq D_H,\quad
u_k\in \mathcal{B}(\delta_1,u^*)\quad\mbox{and}\quad \|\mathcal{F}_k\|\leq\|\mathcal{F}_0\|.\label{eq-th2-2}
\end{eqnarray}
\end{lemma}
\begin{proof}
Since
$$\gamma\|F(u_0)\|^2\leq\gamma(1+c)^2\|u_0-u^*\|^2\leq\min\{\bar{\mu},\eta\},$$
then
$\mu_k\leq\min\{\bar{\mu},\eta\}$ by the updating method of $\mu_k$ in s-Anderson(m) for $k\geq0$. From \eqref{eq-eta}, we find that the first two relations in \eqref{eq-th2-2} hold.

Then, by Theorem \ref{thm1}-(i) and (iii), we have
\begin{equation}\label{eq-mu}
\mathcal{G}(u^*,\mu_k)=G(u^*)=u^*\;\mbox{and}\; \|\mathcal{G}({u},\mu_{k})
-\mathcal{G}({v},\mu_{k})\|\leq c\|u-v\|, \quad \forall k\geq0,\,u,v\in \mathcal{B}(\delta_1,u^*).
\end{equation}
We next prove the last two statements of \eqref{eq-th2-2} by induction, where we see that they are true for $k=0$ and we suppose both of them hold for $0\leq k\leq K$.

Owning to \eqref{eq-mu}, we have
$$\begin{aligned}
\left\|{u}_{K+1}-u^*\right\|=&\left\|\sum_{j=0}^{m_K}\alpha_j^K\mathcal{G}({u}_{K-m_K+j},\mu_{K-m_K+j})
-\sum_{j=0}^{m_K}\alpha_j^K\mathcal{G}({u}^*,\mu_{K-m_K+j})\right\|\\
\leq&M_{\alpha}c\max_j\|{u}_{K-m_K+j}-u^*\|
\leq\frac{M_{\alpha}c}{1-c}\max_j\|\mathcal{F}_{K-m_K+j}\|\\
\leq&\frac{M_{\alpha}c}{1-c}\|\mathcal{F}_{0}\|\leq \frac{M_{\alpha}c(1+c)}{1-c}\|u_{0}-u^*\|,
\end{aligned}$$
which gives $u_{K+1}\in\mathcal{B}(\delta_1,u^*)$ by the condition of $\delta_0$. Then, the third result in \eqref{eq-th2-2} holds for $k=K+1$.

Similarly,
$\sum_{j=0}^{m_K}\alpha_j^K{u}_{K-m_K+j}\in\mathcal{B}(\delta_1,u^*)\subseteq D$. Formulas \eqref{uk} and \eqref{eq-mu} imply
\begin{eqnarray}\label{eq14}
&&\|\mathcal{F}_{K+1}\|=
\|\mathcal{G}(u_{K+1},\mu_{K+1})-u_{K+1}\|\nonumber\\
&&\leq c\|u_{K+1}-\sum_{j=0}^{m_K}\alpha_j^K{u}_{K-m_K+j}\|
+\|{\mathcal{G}}(\sum_{j=0}^{m_K}\alpha_j^K{u}_{K-m_K+j},\mu_{K+1})-
\sum_{j=0}^{m_K}\alpha_j^K\mathcal{G}(u_{K-m_K+j},\mu_{K-m_K+j})\|\nonumber\\
&&\quad  \quad\leq  c\|\mathcal{F}_K\|
+A_K+B_K,
\end{eqnarray}
where
$$A_K=\left\|{\mathcal{G}}(\sum_{j=0}^{m_K}\alpha_j^K{u}_{K-m_K+j},\mu_{K+1})
-\sum_{j=0}^{m_K}\alpha_j^K
{\mathcal{G}}({u}_{K-m_K+j},\mu_{K+1})\right\|,$$
$$B_K=\left\|\sum_{j=0}^{m_K}\alpha_j^K
{\mathcal{G}}({u}_{K-m_K+j},\mu_{K+1})-\sum_{j=0}^{m_K}\alpha_j^K{\mathcal{G}}({u}_{K-m_K+j},\mu_{K-m_K+j})\right\|.$$
Then, by \eqref{eq-G}, we estimate $\|A_K\|$ by the same way as in \cite{ChenKelley2015,TothKelley2015} to get
\begin{equation}\label{eq9}
\begin{aligned}
\|A_{K}\|=&\left\|\Delta_{\sum_{j=0}^{m_K}\alpha_j^K{u}_{K-m_K+j}}-
\sum_{j=0}^{m_K}\alpha_j^K\Delta_{u_{K-m_K+j}}\right\|\\
\leq&\frac{a(M_{\alpha}^2+M_{\alpha})}{2}\max_{j}\|u_{K-m_K+j}-u^*\|^2\\
\leq&\frac{a(M_{\alpha}^2+M_{\alpha})}{2(1-c)^2}\max_{j}\|\mathcal{F}_{K-m_K+j}\|^2\\
\leq&\frac{a(M_{\alpha}^2+M_{\alpha})(1+c)\|u_{0}-u^*\|}{2(1-c)^2}\|\mathcal{F}_{0}\|.
\end{aligned}
\end{equation}
To evaluate $\|B_K\|$, by Theorem \ref{thm1}-(ii), \eqref{smoothing3}, \eqref{smoothing2} and \eqref{eq-mu}, we have
\begin{equation}\label{eq8}
\|B_K\|\leq M_{\alpha}\kappa(\mu_{K-m_K}+\mu_{K+1})\leq 2M_{\alpha}\kappa\mu_0=2M_{\alpha}\kappa\gamma\|F(u_0)\|^2\leq
\frac{2M_{\alpha}\kappa\gamma(1+c)^2\|u_0-u^*\|}{1-c}\|\mathcal{F}_0\|.
\end{equation}
Together \eqref{eq14}, \eqref{eq9}, \eqref{eq8} with the assumption of \eqref{eq-th2-2} for $k=K$, gives
$$
\|\mathcal{F}_{K+1}\|\leq (c+\varpi\|u_0-u^*\|)\|\mathcal{F}_0\|.
$$
Then the fourth relation in \eqref{eq-th2-2} holds for $k=K+1$ by $\delta_0$ satisfying $c+\varpi\delta_0\leq1$. We complete the proof for \eqref{eq-th2-2}.
\end{proof}
\begin{theorem}\label{th3}
Suppose Assumption \ref{ass-Q-H} and Assumption \ref{ass2-2} hold. If $u_0$ is sufficiently close to $u^*$, then the sequence $\{u^k\}$
generated by s-Anderson(m) in Algorithm \ref{algo-o1} converges to the solution of \eqref{eq1-0}
with the r-linear convergence rates of
\begin{equation}\label{sG1-1}
\limsup_{k\rightarrow\infty}\left(\frac{\|u_k-u^*\|}{\|u_0-u^*\|}\right)^{{1}/{k}}\leq {c}\quad\mbox{and}\quad
\limsup_{k\rightarrow\infty}\left(\frac{\|F(u_k)\|}{\|F(u_0)\|}\right)^{{1}/{k}}\leq {c}.
\end{equation}
\end{theorem}
\begin{proof}
Let $\|u_0-u^*\|\leq\delta_0$ with $\delta_0$ in \eqref{eq-u0}. Then, $\mu_0\leq\bar{\mu}$. By the updating method
of $\mu_k$, there exist
$K$ and $\hat{\mu}\in[\epsilon,\bar{\mu}]$ such that $\mu_k=\hat{\mu}$, for all $k\geq K$.

By Lemma \ref{lemma1}, as $\|u_0-u^*\|\leq\delta_0$, we have
$\|u_K-u^*\|\leq\delta_1$. Then, by \eqref{eq_mu2}, we have
$$
\limsup_{k\rightarrow\infty}\left(\frac{\|\mathcal{F}(u_k,\mu_k)\|}{\|\mathcal{F}(u_K,\mu_K)\|}
\right)^{{1}/{(k-K)}}\leq {c},$$
which implies
\begin{equation}\label{sG1}
\limsup_{k\rightarrow\infty}\left(\frac{\|\mathcal{F}(u_k,\mu_k)\|}{\|\mathcal{F}(u_0,\mu_0)\|}
\right)^{{1}/{k}}\leq {c}.
\end{equation}
From \eqref{smoothing2}, we obtain
$$\frac{\|\mathcal{F}(u_k,\mu_k)\|}{\|\mathcal{F}(u_0,\mu_0)\|}\geq\left(\frac{1-{c}}{1+{c}}\right)\frac{\|u_k-u^*\|}{\|u_0-u^*\|},$$
which combines with (\ref{sG1}) and $\limsup_{k\rightarrow\infty}\left(\frac{1-{c}}{1+{c}}\right)^{1/k}=1$ gives the first estimation in \eqref{sG1-1}.
In light of \eqref{smoothing3} and the first relation in (\ref{sG1-1}),
we further obtain the second result in \eqref{sG1-1}.
\end{proof}

From the updating method of $\{\mu_k\}$ in s-Anderson(m), it is a case that $\lim_{k\rightarrow\infty}\mu_k>\epsilon$, which means that there exists $K$ such that $\|F_k\|\leq\sigma_1\|F_{k-1}\|$, $\forall k\geq K$. Combining this with Theorem \ref{th3}, we note that if $\lim_{k\rightarrow\infty}\mu_k>\epsilon$, then s-Anderson(m) not only owns the r-linear convergence in
\eqref{sG1}, but also has the q-linear convergence on residual $\|F(u_k)\|$ with factor $\sigma_1$.
Moreover,
following the statements in Theorem \ref{th3}, even if we have no knowledge on $\bar{\mu}$ and $\eta$, the local convergence properties of s-Anderson(m) in Theorem \ref{th3} are always valid with any $\sigma_1,\sigma_2\in(0,1)$ by setting $\epsilon$ sufficiently small.
In particular, if $\mu_0$ is sufficiently small such that $\mu_k$ is unchanged in s-Anderson(m), then s-Anderson(m) is just Anderson(m) on $\mathcal{G}(u,\mu_0)$. A simple consideration is that the results in Theorem \ref{th3} also hold if we let $\mu_k:=\epsilon$ with $\epsilon$ being sufficiently small. Similar results in
Theorem \ref{th3} also hold for the EDIIS(m) with the same smoothing approach.
\begin{remark}\label{remark1}
According to Rademacher's theorem, a locally Lipschitz continuous function $G$ is differentiable almost everywhere.
If $\psi$ is a smoothing function of $\max\{t,0\}$,
Proposition \ref{p4_3} says that the contraction factor of $\mathcal{G}(\cdot,\mu)$ on $D$ can be sufficiently close to the contraction factor of
${G}$ as $\mu$ is sufficiently small. Theorem \ref{thm1} gives an upper bound of the contraction factor of $\mathcal{G}(\cdot,\mu)$ on $D$ with $\psi$ defined in \eqref{smoothing1}.
By the structure of $\psi$ in \eqref{smoothing1}, if $G$ is not
continuously differentiable at $u^*$,
which means that there is $i\in\{1,\ldots,l\}$ such that $Q_i(u^*)=\underline{w}_i$ or $\overline{w}_i$, then
the contraction factor of $\mathcal{G}(\cdot,\mu)$ with \eqref{smoothing1}
can be strictly smaller than the  contraction factor of $G$ around $u^*$ as $\mu$ is smaller than a threshold.

For example, if $G(u)=(\max\{u_1/2,0\},1-{u_2}/{4})^{\rm T}$,
the exact contraction factor of $G$ around its fixed point $u^*=(0,{4}/{5})^{\rm T}$ is ${1}/{2}$. Let $\mathcal{G}(u,{\mu})=(\psi(u_1/2,\mu),1-u_2/4)^{\rm T}$ with the definition of $\psi$ in \eqref{smoothing1}.
For any given ${\mu}\in[\epsilon,1]$,
we note that
\[\|\mathcal{G}'(u,{\mu})\|\leq\max\{|u_1|/(2\mu),1/4\},
\quad \forall u\in \mathcal{B}(\delta,u^*)\]
with $\delta\leq\epsilon/2$,
which implies that the contraction factor of $\mathcal{G}(\cdot,{\mu})$ is no larger than $1/4$ when ${\mu}\in[\epsilon,1]$.
These results combining the analysis in Theorem \ref{th3} show that as $u_0$ is sufficiently close to $u^*$, s-Anderson(m) is r-linearly convergent to the fixed point of $G$ with factor no larger than
$1/4$, which is strictly smaller than the contraction factor of $G$ around $u^*$. And the contraction factor of $\mathcal{G}(u,\mu)$ on $\mathcal{B}(\delta,u^*)$ is decreasing as $\mu$ is increasing in $[\epsilon,1]$.
\end{remark}
\section{Numerical applications and examples}\label{section4}
In this section, we illustrate our new convergence results of Anderson acceleration for nonsmooth fixed point problem \eqref{eq1-0} by three applications.
All the numerical experiments are performed in MATLAB 2016a on a Lenovo PC547 (3.00GHz, 2.00GB of RAM). When $m\geq2$, proceeding as in \cite{Fang,Walker}, we write the problem in \eqref{Anderson_alpha} by the following equivalent form
\begin{equation}\label{Anderson_theta}
\theta^k\in\arg\min_{\theta\in\mathbb{R}^{m_k}}\,\left\|F_k-\sum\nolimits_{j=0}^{m_k-1}\theta_j({F}_{k-m_k+j+1}-F_{k-m_k+j})\right\|
\end{equation}
and then
$$u_{k+1}=G(u_k)-\sum\nolimits_{j=0}^{m_k-1}\theta_j^k({G(u_{k-m_k+j+1})
-G(u_{k-m_k+j})})),$$
in terms of the original iterations $\alpha_j^k$ in \eqref{Anderson_alpha}, where
$\alpha_0^k=\theta_0^k$, $\alpha_j^k=\theta_j^k-\theta_{j-1}^k$ for $1\leq j\leq m_k-1$ and $\alpha_{m_k}^k=1-\theta_{m_k-1}^k$.
To solve \eqref{Anderson_theta}, we consult the method based on the pseudoinverse introduced in \cite{Fang}, and it has been shown that the deteriorating condition of the least-squares matrix does not necessarily interfere with convergence \cite{TothKelley2015}. This method is also used to find the $\alpha^k_j$ in s-Anderson(m) in Algorithm \ref{algo-o1}. For s-Anderson(m), we always set $\epsilon=10^{-10}$, $\gamma=1/n$
and $\sigma_1=\sigma_2=0.6$ for comparison.
And we stop Anderson(m) in Algorithm \ref{algo-o} and s-Anderson(m) in Algorithm \ref{algo-o1} when
\begin{equation}\label{stop}
\frac{\|F(u_k)\|}{\|F(u_0)\|}\leq 10^{-14} \quad {\rm or } \quad k\geq7000.
\end{equation}

It should be noticed that
the stopped criterion for s-Anderson(m) also uses the value of $\|F(u_k)\|/\|F(u_0)\|$
not the residual on smoothing approximation $\mathcal{F}(u,\mu)$.
   From these numerical results in Examples \ref{example1}-\ref{example3}, we have the following observations.
\begin{itemize}
\item [\rm (i)] Both Anderson(m) and s-Anderson(m) can be used to solve the considered problems, in which the contraction mappings $G$ are nonsmooth at the fixed points. Though the theoretical results of them are built up for local convergence, it is satisfactory that all the numerical experiments in this section are convergent with random initial points.
\item[\rm (ii)]  For both Anderson(m) and s-Anderson(m), as presented in the experiments, the best choice of $m$ is problem dependent.
\item[\rm (iii)] s-Anderson(m) performs better than Anderson(m) for most cases, and the local convergence of $\|F(u_k)\|/\|F(u_0)\|$ by s-Anderson(m) is also faster. Since the mapping $\mathcal{G}(u,\mu)$ used in s-Anderson(m) only has small difference with $G(u)$ in Anderson(m), the generated $u_k$ in the former iterations cannot bring obvious differences on $\|F_k\|/\|F_0\|$ when $\|F_k\|$ is relatively large. However, after certain iterations, $\|F_k\|$ is reduced significantly and the advantages of s-Anderson(m) appears clearly. So it is reasonable that s-Anderson(m) outperforms Anderson(m) when the accuracy is high.
\item[\rm (iv)]  The superiorities of s-Anderson(m) over Anderson(m) become more and more obvious as the number of elements in $\{i:Q_i(u^*)=\underline{w}_i\,\,\mbox{or}\,\,\overline{w}_i\}$ increases.
\end{itemize}
\subsection{Minimax optimization problem}
Constrained minimax optimization problem is often modeled by
\begin{equation}\label{eq1-ex1}
\min_{x\in\mathcal{X}}\max_{y\in\mathcal{Y}} f(x,y),
\end{equation}
where $f:\mathcal{X}\times\mathcal{Y}\rightarrow\mathbb{R}$ is a convex-concave function over closed, convex sets $\mathcal{X}\subseteq\mathbb{R}^{n_1}$ and $\mathcal{Y}\subseteq\mathbb{R}^{n_2}$.
Such models are widely used in game theory, machine learning and parallel computing. Due to the convexity and concavity of $f$ with respect to $x$ and $y$, respectively, $((x^*)^{\rm{T}},(y^*)^{\rm{T}})^{\rm{T}}$ is a saddle point of \eqref{eq1-ex1}, if and only if it satisfies
\begin{equation}\label{eq2-ex1}
\left\{\begin{aligned}
&x^*=P_{\mathcal{X}}(x^*-\alpha\nabla_xf(x^*,y^*))\\
&y^*=P_{\mathcal{Y}}(y^*+\beta\nabla_yf(x^*,y^*))
\end{aligned}
\right.
\end{equation}
with $\alpha,\beta>0$. Denote
$$u=\left(
      \begin{array}{c}
        x \\
        y \\
      \end{array}
    \right),\;\Lambda=\left(
                        \begin{array}{cc}
                          \alpha I_{n_1} & 0 \\
                          0 & \beta I_{n_2} \\
                        \end{array}
                      \right),\,
    L(u):=L(x,y)=\left(
                           \begin{array}{c}
                            \nabla_x f(x,y) \\
                             -\nabla_y f(x,y) \\
                           \end{array}
                         \right),\;
                         \Omega=\mathcal{X}\times\mathcal{Y}.
    $$
Then, \eqref{eq2-ex1} is expressed by
$u^*=P_{\Omega}(u^*-\Lambda L(u^*)),$
which is reduced to a fixed point problem of $G$ with
\begin{equation}\label{eq3-ex1}
G(u):=P_{\Omega}(u-\Lambda L(u)).
\end{equation}
The mapping in \eqref{eq3-ex1} can be formulated by \eqref{eq1-0} with
$Q(u)=u-\Lambda L(u)$ and $H(v)=v$.
\begin{assumption}\label{ass2}
The mapping $L$ is strongly monotone and Lipschitz continuous, i.e. there exist positive parameters $\tau_L$ and $c_L$ such that for all $u,\tilde{u}\in\Omega$, it holds
\begin{eqnarray*}
&(L(u)-L(\tilde{u}))^{\rm{T}}(u-\tilde{u})\geq\tau_L\|u-\tilde{u}\|^2,&\\
&\|L(u)-L(\tilde{u})\|\leq c_L\|u-\tilde{u}\|.
\end{eqnarray*}
\end{assumption}

For $u,\tilde{u}\in\Omega$, by the Lipschitz property of $P_{\Omega}$ and Assumption \ref{ass2}, when $\alpha=\beta$, we obtain
$$\begin{aligned}
&\|P_{\Omega}(u-\alpha L(u))-P_{\Omega}(\tilde{u}-\alpha L(\tilde{u}))\|^2\\
\leq&\|u-\alpha L(u)-\tilde{u}+\alpha L(\tilde{u})\|^2\\
=&\|u-\tilde{u}\|^2+\alpha^2\|L(u)-L(\tilde{u})\|^2-2\alpha(u-\tilde{u})^T(L(u)-L(\tilde{u}))\\
\leq&\left(1+\alpha^2c_L^2-2\alpha\tau_L\right)\|u-\tilde{u}\|^2.
\end{aligned}$$
It is easy to verify that $1+\alpha^2c_L^2-2\alpha\tau_L\in(0,1)$, if $\alpha\in\left(0,{2\tau_L}/{c_L^2}\right)$. Hence under Assumption \ref{ass2}, if $\alpha=\beta\in\left(0,{2\tau_L}/{c_L^2}\right)$, then
$G$ in \eqref{eq3-ex1} is a contractive mapping with factor $c=\sqrt{1+\alpha^2c_L^2-2\alpha\tau_L}$ and the conclusions in Theorem \ref{th3} hold for $G$ in \eqref{eq3-ex1}, which prompts us to find the fixed point of $G$ by using s-Anderson(m)
with the smoothing approximation of $G$
defined in \eqref{G-s}. To show the effectiveness of the corresponding theoretical results and the effect of s-Anderson(m) on solving problem
\eqref{eq1-ex1}, we conduct the numerical experiment on a special case of \eqref{eq1-ex1}, which comes from the two-payers Nash game problems.
\begin{example}\label{example1}
Consider
\begin{equation}\label{eq6-ex1}
\min_{x\in \mathbb{R}^{n_1}_+}\max_{y\in \mathbb{R}^{n_2}_+} f(x,y):=\frac{1}{2}x^{\rm{T}}Ax+x^{\rm{T}}By-\frac{1}{2}y^{\rm{T}}Cy+ a^{\rm T}x-b^{\rm{T}}y,
\end{equation}
where $A\in\mathbb{R}^{n_1\times n_1}$ and $C\in\mathbb{R}^{n_2\times n_2}$ are symmetric positive definite matrices, $B\in\mathbb{R}^{n_1\times n_2}$, $a\in\mathbb{R}^{n_1}$ and $b\in \mathbb{R}^{n_2}$
are random matrix and vectors.
Denote $\lambda_{\min}(A)$ and $\lambda_{\min}(C)$ the minimal eigenvalues of $A$ and $C$, respectively.
Let $\Omega=\mathbb{R}^{n_1+n_2}_+$ and $L(u)=M u+d$ with $u=(x^{\rm{T}},y^{\rm{T}})^{\rm{T}}$, $M=\left(
                                                              \begin{array}{cc}
                                                                A & B \\
                                                                -B^{\rm{T}} & C \\
                                                              \end{array}
                                                            \right)$ and $d=\left(
                                                                \begin{array}{c}
                                                                  a \\
                                                                  b \\
                                                                \end{array}
                                                              \right)$, which satisfies Assumption \ref{ass2} with
\begin{equation}\label{eq-ex2-1}
\tau_L=\min\{\lambda_{\min}(A),\lambda_{\min}(C)\}\quad\mbox{and}\quad c_L=\|M\|.
\end{equation}

Based on the above analysis, the solution of \eqref{eq6-ex1} can be transformed to the fixed point of \eqref{eq3-ex1}, and
when we choose
\begin{equation}\label{eq-ex2-2}
\alpha=\beta={\tau_L}/{c_L^2},
\end{equation}
$G$ in \eqref{eq3-ex1} is a
contractive mapping with factor $c=\|I_{n_1+n_2}-\alpha M\|$.
For given positive integers $n_1=1000$, $n_2=500$ and $s_1=0.3$, we generate matrices $A$, $C$ and $B$ as follows:
\begin{eqnarray*}&{\tt A1= 1 + a*rand(n_1,1);
U1 = orth(rand(n_1,n_1));
A = U1' * diag(A1) * U1;}&\\
&{\tt C1= 2 + b*rand(n_2,1);
U2 = orth(rand(n_2,n_2));
C = U2' * diag(C1)* U2;} &\\
&{\tt B=sprand(n_1,n_2,s_1);B=full(B)/norm(B);}&
\end{eqnarray*}
Then, we set $n=n_1+n_2$, $\alpha$ and $\beta$ be defined by \eqref{eq-ex2-2} with the parameters in
\eqref{eq-ex2-1}. It is clear that $G$ in \eqref{eq3-ex1} is nonsmooth at $u^*$ if there exists $i$ such that $(Mu^*+d)_i=0$ and $u_i^*=0$. So, for given $s_2=0.5$, we generate the fixed point $u^*$ ({\tt sol} in the code) with $n\times s_2$ elements of $0$ and vector $d\in\mathbb{R}^n$ such that the corresponding elements of $Mu^*+d$ are also $0$ by the following codes:
\begin{eqnarray*}
&{\tt index = randperm(n);index1 = index(1:s_2*n);
}&\\
&{\tt sol=0.1+0.9*rand(n,1);
sol(index1)=0;
M=[A\quad B; -B'\quad C];
d=-M*sol;}&
\end{eqnarray*}

Let $u_0={\rm zeros}(n,1)$. For different values of $a$ and $b$, which influence the contractive factor of $G$ in \eqref{eq3-ex1},
the number of iterations of Anderson(m) and s-Anderson(m) to find $u_k$ satisfying (\ref{stop}) are shown in Table \ref{table2}, where the values are the mean values of $50$ random experiments. From Table \ref{table2}, we see that though
the contractive factors of $G$ are all very close to $1$, both Anderson(m) and s-Anderson(m) work well, and s-Anderson(m) performs better for most cases. Throughout the whole table, the smallest iterations for all cases are presented by s-Anderson(m) with $m=3$ or $m=5$. Fig. \ref{fig2} plots the convergence behaviors of s-Anderson(1) and s-Anderson(3) with some different values of $\sigma_1=\sigma_2$, where the best is located at $\sigma_1=\sigma_2=0.6$. This is an interesting thing that we can let the value of $\epsilon$ be sufficiently small to guarantee the efficiency of s-Anderson(m), and control the values of $\sigma_1$ and $\sigma_2$ to improve its convergence behaviours. How to choose better parameters is an interesting topic for further study.
\begin{table}[htbp]
\centering{
\begin{tabular}{|c|c||c|c|c|c|c|c|} \hline
\multicolumn{2}{|c||}{Parameters}&\multicolumn{6}{|c|}{Anderson(m)/s-Anderson(m)}\\
\hline
$a,b$&$c$&$m=0$&$m=1$&$m=2$&$m=3$&$m=5$&$m=10$\\
\hline
$0,0$&0.835 & 150/\textbf{148} & 72/\textbf{65} & 60/\textbf{57} & 58/\textbf{48} & 62/\underline{\textbf{43}} & 78/\textbf{54}\\
$2,1$&0.893 & 230/\textbf{218} & \textbf{68}/85 & 71/71 & 71/\textbf{65} & 80/\underline{\textbf{58}} & 80/\textbf{66}\\
$1,1$&0.895 & 246/\textbf{236} & \textbf{74}/92 & 76/\textbf{73} & 77/\underline{\textbf{65}} & 84/\textbf{71} & 94/\textbf{80}\\
$1,2$&0.943 & 465/\textbf{446} & 147/\textbf{129} & 113/\textbf{105} & 117/\underline{\textbf{102}} & 125/\textbf{107} & 159/\textbf{126}\\
$3,1$&0.941& 407/\textbf{379} & 114/\textbf{104} & 105/\textbf{96} & 108/\underline{\textbf{93}} & 111/\textbf{94} & 122/\textbf{102}\\
$3,3$&0.961 &  609/\textbf{565} & 194/194 & 136/\textbf{123} & 144/\textbf{123} & 149/\underline{\textbf{122}} & 174/\textbf{139}\\
\hline
\end{tabular}}\caption{Numerical results of Anderson(m) and s-Anderson(m) for Example \ref{example1}}\label{table2}
\end{table}
\begin{figure}
\begin{center}
  \subfigure[s-Anderson(1)]{
    \label{fig2-1}
    \includegraphics[width=2.2in,height=2.0in]{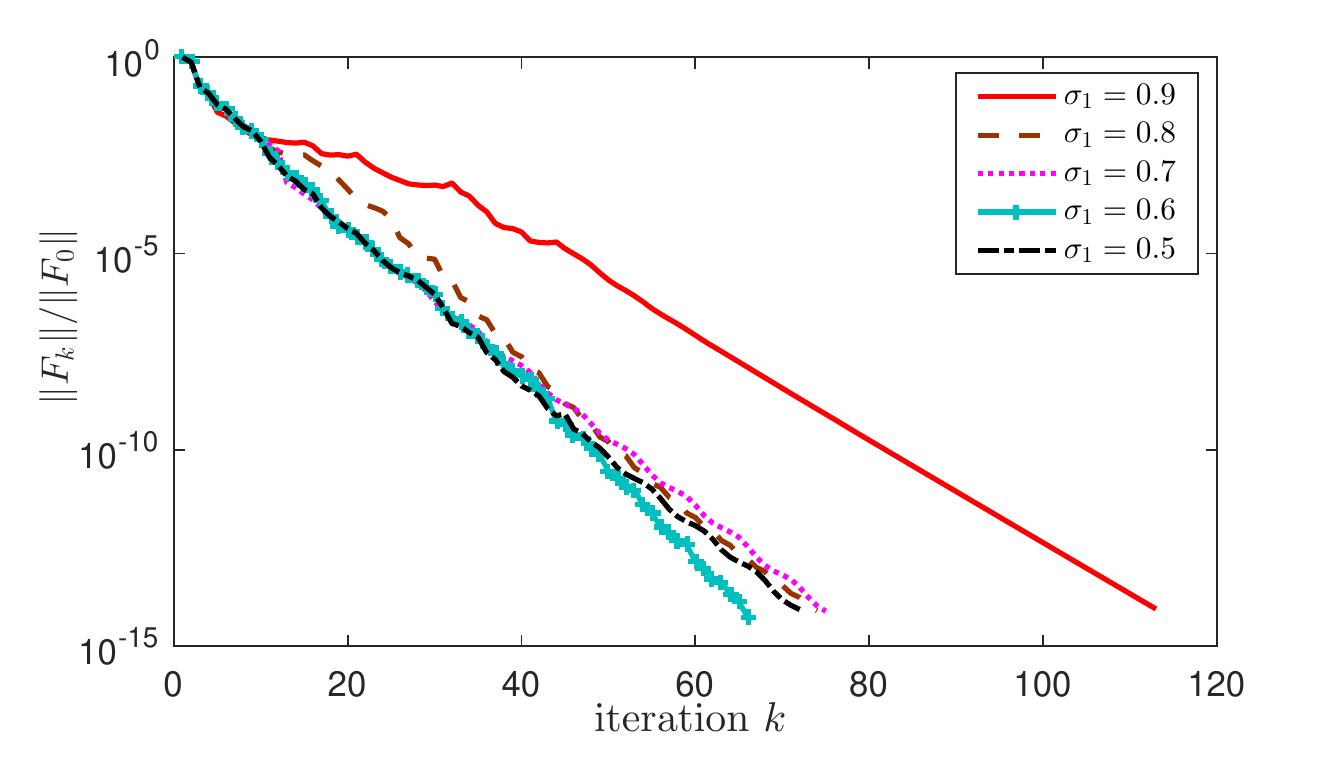}}
  \subfigure[s-Anderson(3)]{
    \label{fig2-2}
    \includegraphics[width=2.2in,height=2.0in]{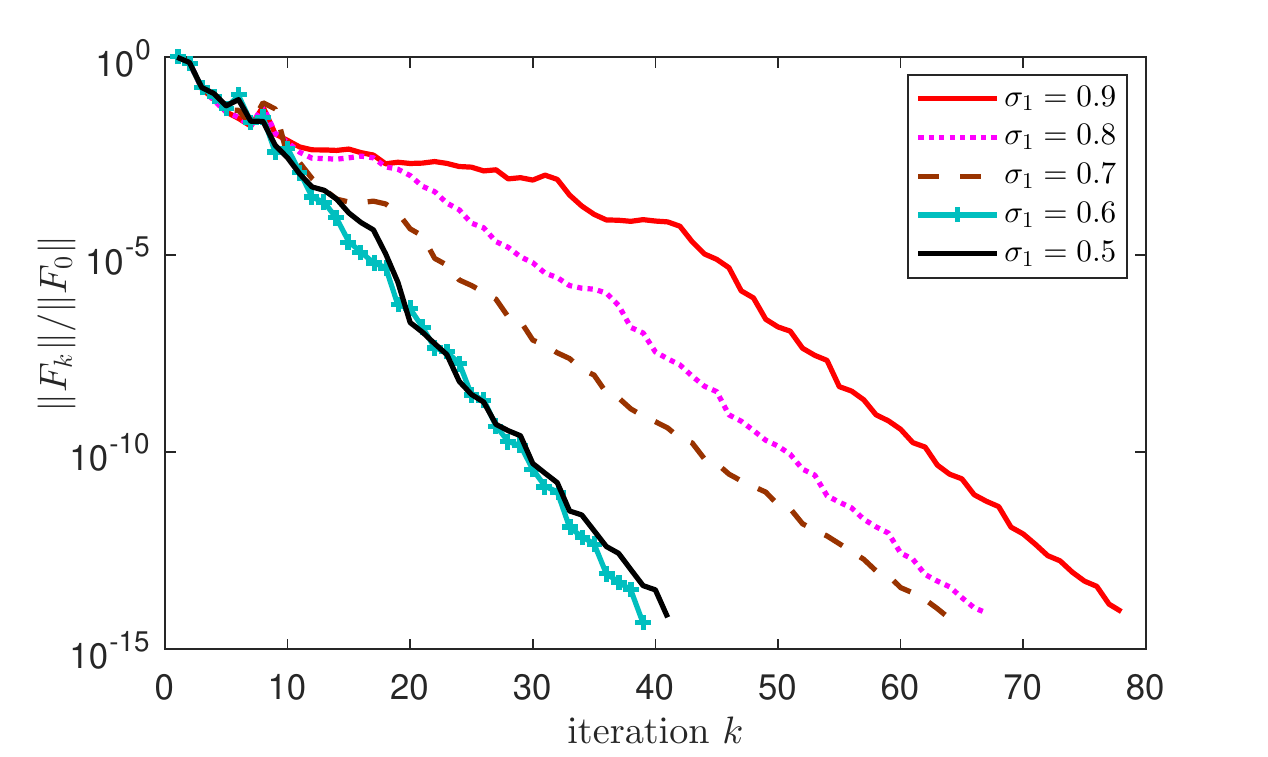}}
\caption{Convergence of $\|F_k\|/\|F_0\|$ by s-Anderson(1) and s-Anderson(3) with different values of $\sigma_1=\sigma_2$ for Example \ref{example1}} \label{fig2}
\end{center}
\end{figure}
\end{example}
\subsection{Complementarity problem}
Given a continuously differentiable function $f: \mathbb{R}^{n} \to \mathbb{R}^{n}$, the  complementarity problem is to find $v$ such that
$$v\ge 0, \quad f(v)\ge 0, \quad v^{\rm T}f(v)=0.$$
This problem is denoted as CP$(f)$, which is equivalent to
$v=\max\{v-f(v),0\}.$
Let $Q(v)=v-f(v)$. If $\|I- f'(v)\|\le c<1$, then $G(v)= \max\{Q(v),0\}$ is a contraction mapping with factor $c$.

If $f(v)=Mv+q$ with $M\in \mathbb{R}^{n\times n}$ and $q\in \mathbb{R}^{n}$, the CP$(f)$ is the linear complementarity problem, denoted as LCP$(q,M)$.
Suppose $M=(m_{ij})_{n\times n}$ is strictly diagonally dominate with positive diagonal elements in the following sense, $$\sum^n_{i=1,i\neq j}|m_{ij}|< m_{ii} \quad {\rm and} \quad \sum^n_{j=1,j\neq i}|m_{ij}|< m_{ii}.$$ Let
$\Lambda=$diag$(m_{ii})$. Then LCP$(q,M)$ is equivalent to
$$\Lambda v\ge 0, \quad M\Lambda^{-1}\Lambda v+q\ge 0, \quad (\Lambda v)^{\rm T}(M\Lambda^{-1}\Lambda v+q)=0$$
and can be solved via LCP$(q, M\Lambda^{-1})$. Moreover, from
\begin{eqnarray*}
\|I-M\Lambda^{-1}\| &\le& \sqrt{\|(\Lambda- M)\Lambda^{-1}\|_1\|(\Lambda- M)\Lambda^{-1}\|_\infty}\\
&=&\sqrt{\max \frac{1}{m_{ii}} \sum_{j=1,j\neq i}^n |m_{ij}|}\sqrt{ \max \frac{1}{m_{ii}} \sum_{i=1,i\neq j}^n |m_{ij}|}=:c<1,
\end{eqnarray*}
$G(u)= \max((I-M\Lambda^{-1})u-q,0)$ is a contraction mapping.
Let $Q(u)=(I-M\Lambda^{-1})u-q$. We define a smoothing approximation of $G$ by \eqref{G-s},
which is also a contraction mapping with factor $c$ and satisfies the conditions in Assumption \ref{ass-Q-H}. Thus, if $u^*$ is the fixed point of the above defined $G$, then $v^*=\Lambda^{-1}u^*$ is the solution of LCP$(q,M)$.
\begin{example}\label{example2}
Pricing American options in a partial differential equation framework with finite difference methods or finite element methods lead to a linear complementarity problem
\begin{equation}\label{pricing}
v-a\ge 0, \quad Mv-b\ge 0, \quad (v-a)^{\rm T}(Mv-b)=0,
\end{equation}
where $v$ is the value of an American option, $a$ is from a given payoff function,  $b$ is from an initial guess of the value and its changing rate, and $M$ is from differential operators  \cite{HJB}.

Let $u=v-a$ and $q=Ma-b$, then \eqref{pricing} is the standard form of LCP($q,M)$.
We set

$$M=\left(\begin{array}{rrrr}
2+\gamma_1 h^2 & -1+0.5h\tau_1 &  & \\
-1-0.5h\tau_2 & 2 +\gamma_2 h^2 & -1 +0.5h\tau_2&  \\
\ddots&\ddots&\ddots &\\
 &   &     2+\gamma_{n-1} h^2 &-1+0.5h\tau_{n-1}\\
   & & -1-0.5h\tau_n & 2+\gamma_n h^2
   \end{array}
   \right).$$
Here, $M$ is the matrix from the centered difference formulate for
$$-\frac{\partial^2V}{\partial x^2}(t,x) +\tau(t,x)\frac{\partial V}{\partial x}(t,x)+ \gamma(t,x)
V(t,x)$$
at a fixed time $t$,  where $h=1/(n+1)$ is the mesh size of discretization, and $\gamma(t,x)>0$ and $\tau(t,x)$ are given functions.
If $|\tau_{i}|=|\tau_{i+1}|<2(n+1)$, $i=1,2,\ldots,n-1$, the matrix $M$ is a strictly diagonal dominate matrix, and thus a P-matrix.
Then, the LCP($q,M)$ has a unique solution $u^*$ for any $q \in \mathbb{R}^n$, which is also the fixed point of the nonsmooth fixed point problem
\begin{equation}\label{eq-ex2}
u=G(u)=\max\{(I-\eta M)u-\eta q,0\},
\end{equation}
with $\eta=\frac{1}{2+\gamma h^2}$. Here function
$G$ in \eqref{eq-ex2} is a contraction mapping with the contraction factor $c=2\eta$ and $G_i$ is not differentiable at the solution $u^*$ for
\begin{equation}\label{eqN}
i\in {\mathcal{N}}:=\{i \, : \, ( (I-\eta M)u^*-\eta q)_i=0\}.
\end{equation}

Throughout this example, we choose $u_0=0.5*{\rm ones}(n,1)$ and set $\gamma(t,x)\equiv 10^3$, $\tau(t,x)\equiv-1$.
For given $n$ and $\Theta\in (0,1)$ ({\tt theta}), we randomly generate the solution $u^*$ ({\tt sol}) and corresponding $q$ as follows
\begin{equation}\label{eq11}{\tt
sol= \max\{{\rm rand}(n,1)- theta, 0\}; \quad  q=-M*sol;}
\end{equation}
By the setting of this problem, there are around $\Theta\times n$ components
in $\mathcal{N}$ defined by \eqref{eqN}.

First, we compare the performance of Anderson(m) and s-Anderson(m) with different values of $m$.
Set $\Theta=0.4$, and $n=200$, $300$ in \eqref{eq11}. The convergence of $\|F_k\|/\|F_0\|$ for Anderson(m) and s-Anderson(m) with $m=0,1,2,3,10$ are plotted in
Fig. \ref{fig5}, from which we can see that s-Anderson(m) is faster than Anderson(m) always and s-Anderson(10) is the best. In \cite{Pollock2}, the following  dynamically updating of depth $m_k$ is introduced and used,
\begin{equation}\label{depth}
m_k={\rm median}([m_1;\tilde{m}_k;m_2])\quad\mbox{with}\quad
\tilde{m}_k={\rm ceil}(-\log_{10}\|F_k\|),
\end{equation}
where $m_1$ and $m_2$ are positive integers to control the lower and upper bounds of $m$.
In particular, if $m_1=m_2$, then the corresponding algorithms are just Anderson(m) and s-Anderson(m) with $m=m_1=m_2$. Fig. \ref{fig6} shows the number of iterations of Anderson(m) and s-Anderson(m) to satisfy the stop criterion in \eqref{stop} using dynamic depth selection \eqref{depth} with $m=m_1=m_2$ and $m_1\neq m_2$, in which the best result is located at $m=m_1=m_2=8$ by s-Anderson(m). From Fig. \ref{fig6}, we find that the number of iterations is not monotone decreasing as $m$ is increasing. Whether the dynamic depth selection approaches can improve the convergence of Anderson acceleration methods is an interesting topic for further research.
\begin{figure}
\begin{center}
  \subfigure[$n=200$]{
    \label{fig3-1}
    \includegraphics[width=2.2in,height=1.8in]{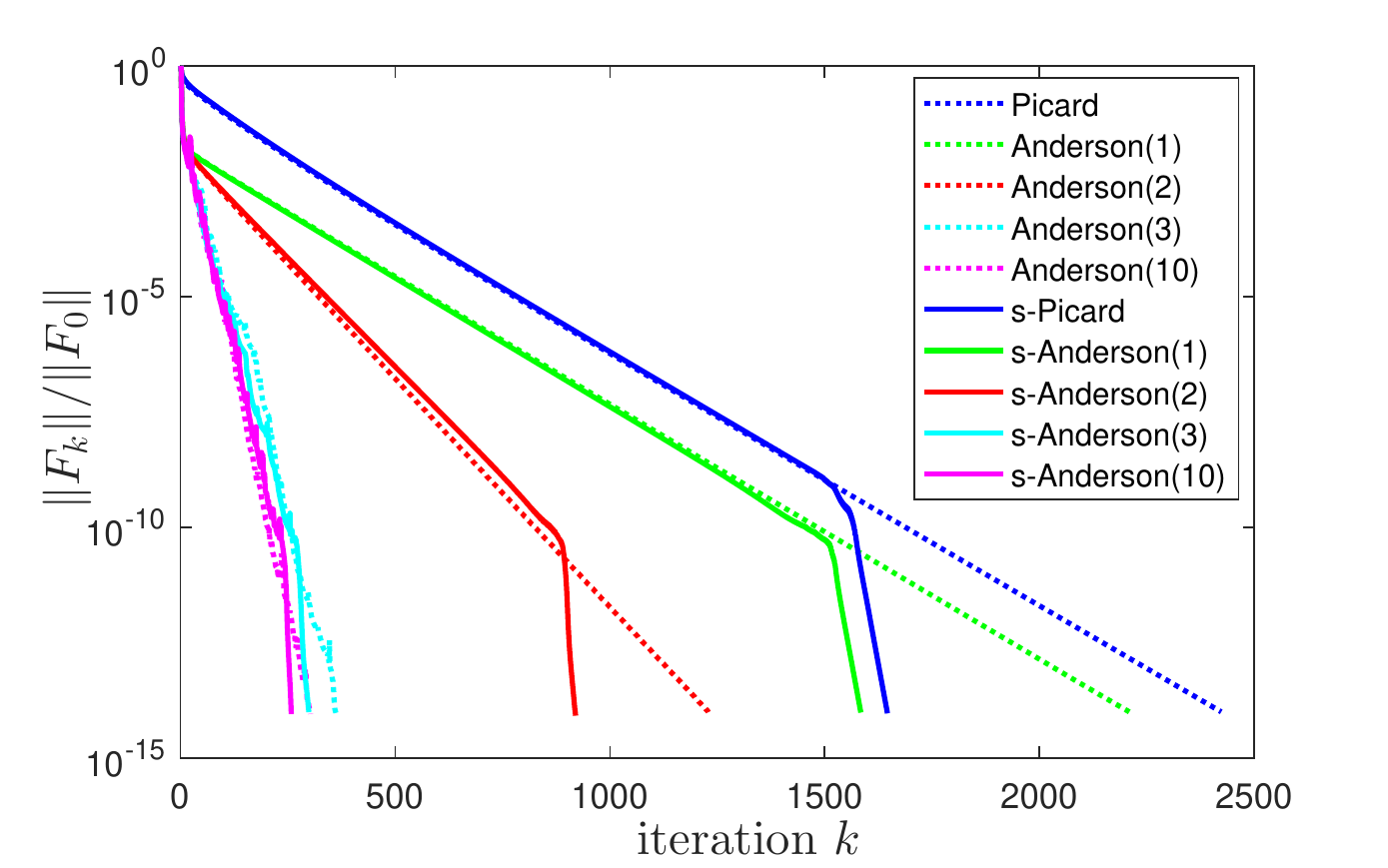}}
  \subfigure[$n=300$]{
    \label{fig3-2}
    \includegraphics[width=2.2in,height=1.8in]{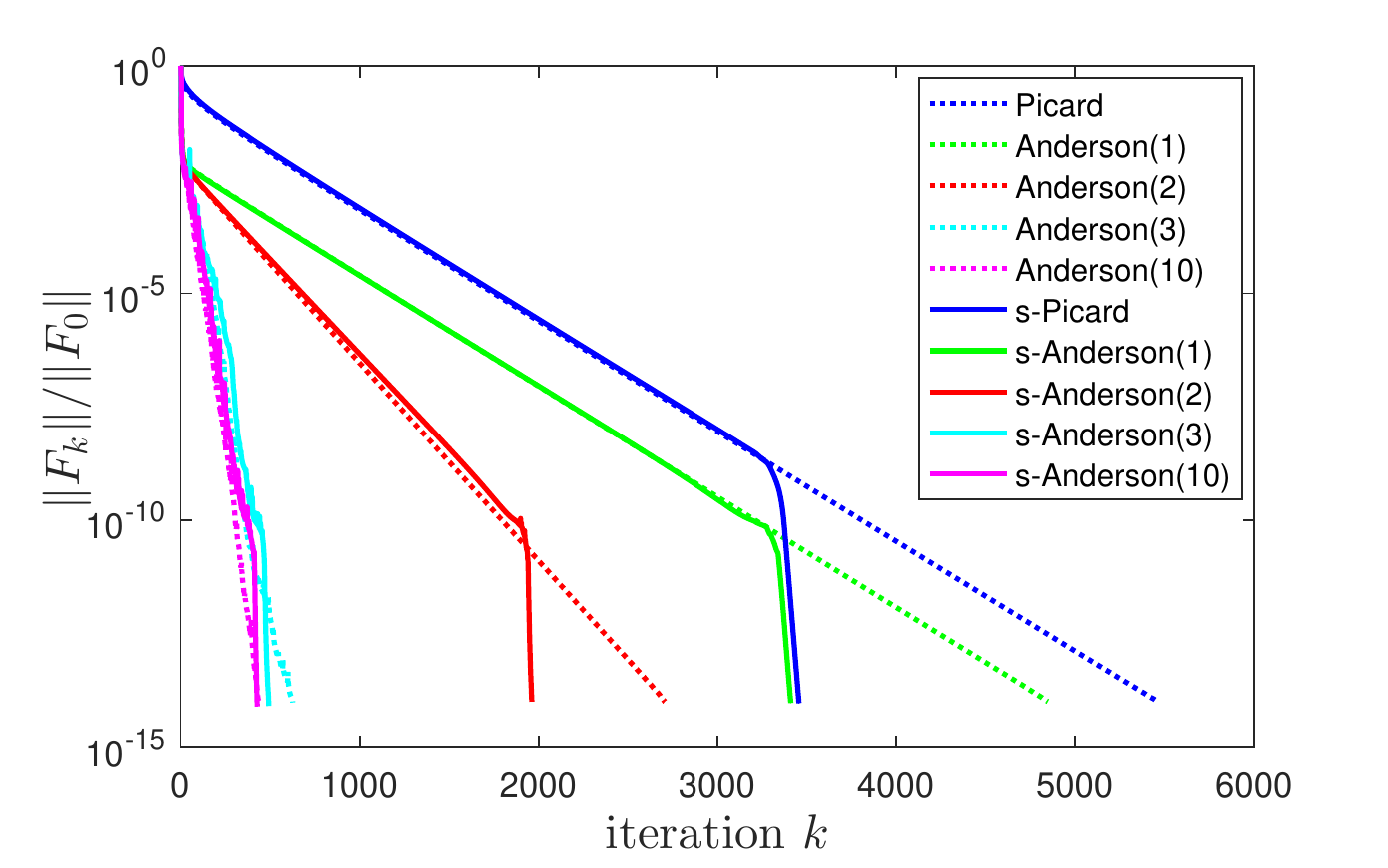}}
\caption{Convergence of $\|F_k\|/\|F_0\|$ by Anderson(m) and s-Anderson(m) for Example \ref{example2} with $n=200$ and $n=300$} \label{fig5}
\end{center}
\end{figure}
\begin{figure}
\begin{center}
\includegraphics[width=5.2in,height=2.0in]{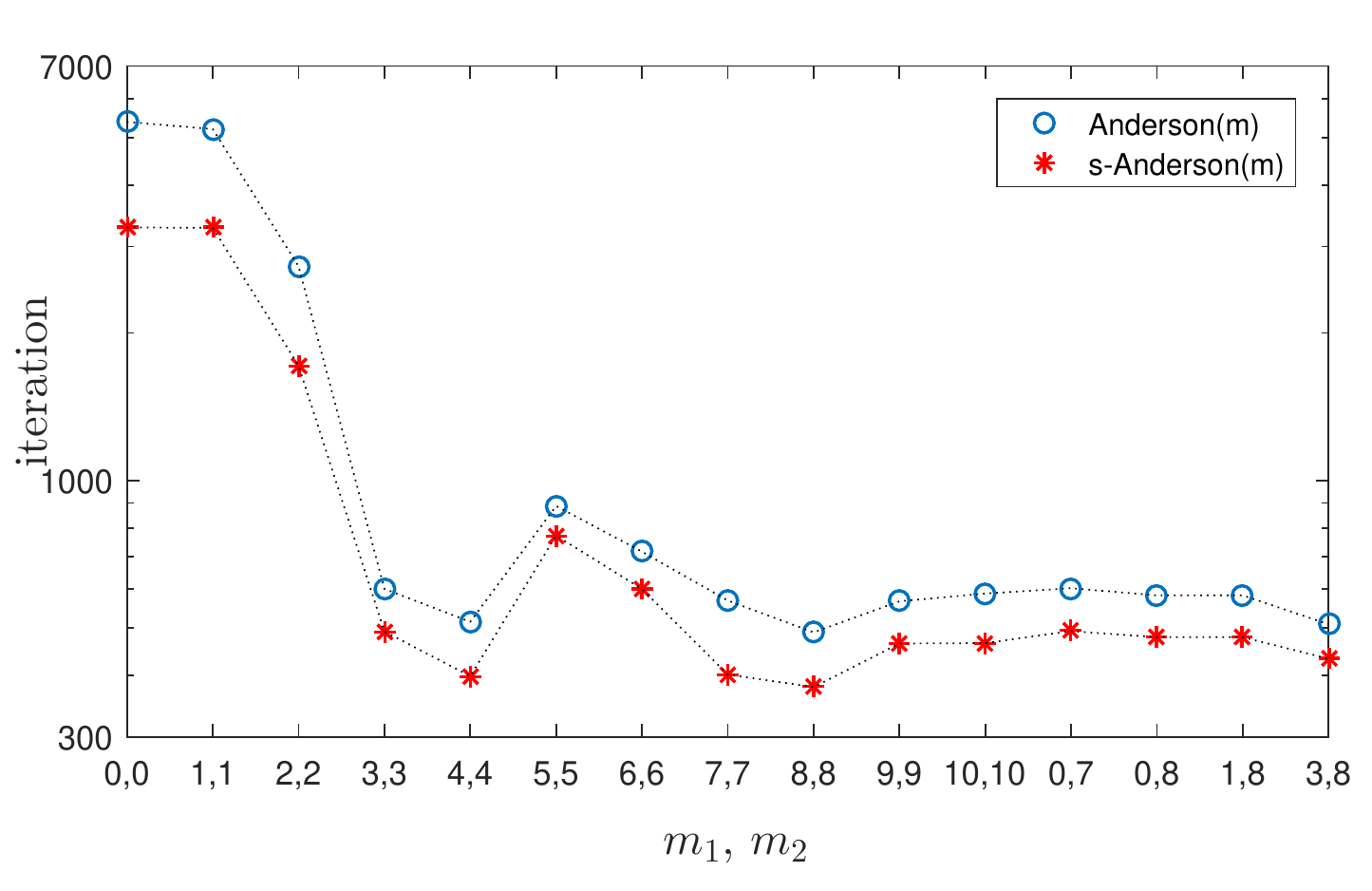}
\caption{Performance of Anderson(m) and s-Anderson(m) using dynamic depth selection \eqref{depth} with $m=m_1=m_2$ and $m_1\neq m_2$ for Example \ref{example2}} \label{fig6}
\end{center}
\end{figure}

Next, we test the performance of Anderson(m) and s-Anderson(m) for different values of $\Theta$,
since its value controls the number of dimensions, on which $G$ is nonsmooth at $u^*$. Let $n=200$.
For $\Theta=0.2$, $0.4$, $0.6$ and $0.8$, we plot the convergence of $\|F_k\|/\|F_0\|$ by Anderson(m) and s-Anderson(m) with $m=1,10$ in Fig. \ref{fig4}.
The displayed results in Fig. \ref{fig4} show that s-Anderson(m) is faster than Anderson(m) for all these cases. In particular, as $\Theta$ is larger, the superiority on the local convergence rate of s-Anderson(m) compared with Anderson(m) is more obvious, which
corresponds to the observation (iv) given at the beginning of this section.
\begin{figure}
\begin{center}
  \subfigure[]{
    \label{fig4-1}
    \includegraphics[width=2.3in,height=1.8in]{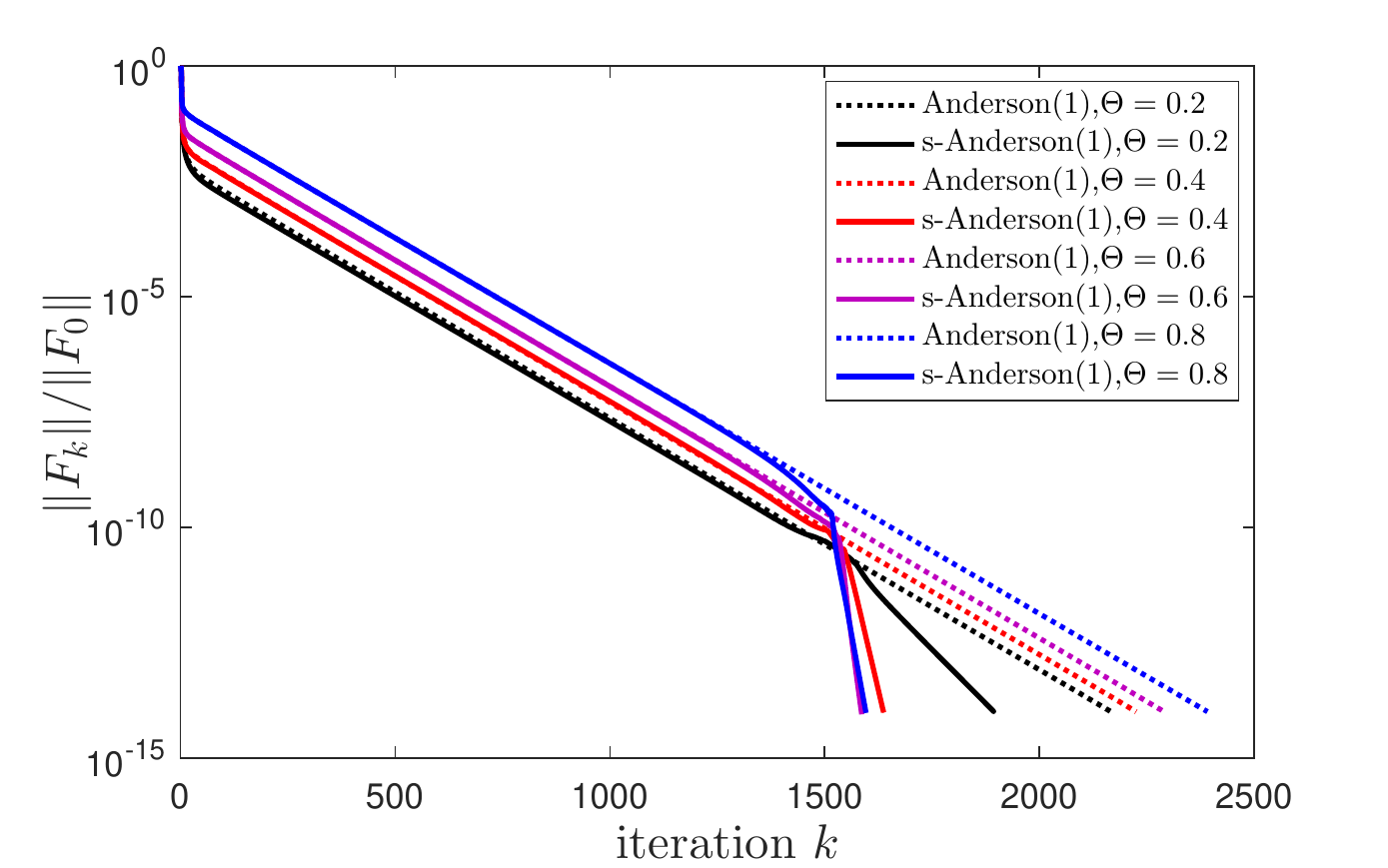}}
  \subfigure[]{
    \label{fig4-2}
    \includegraphics[width=2.3in,height=1.8in]{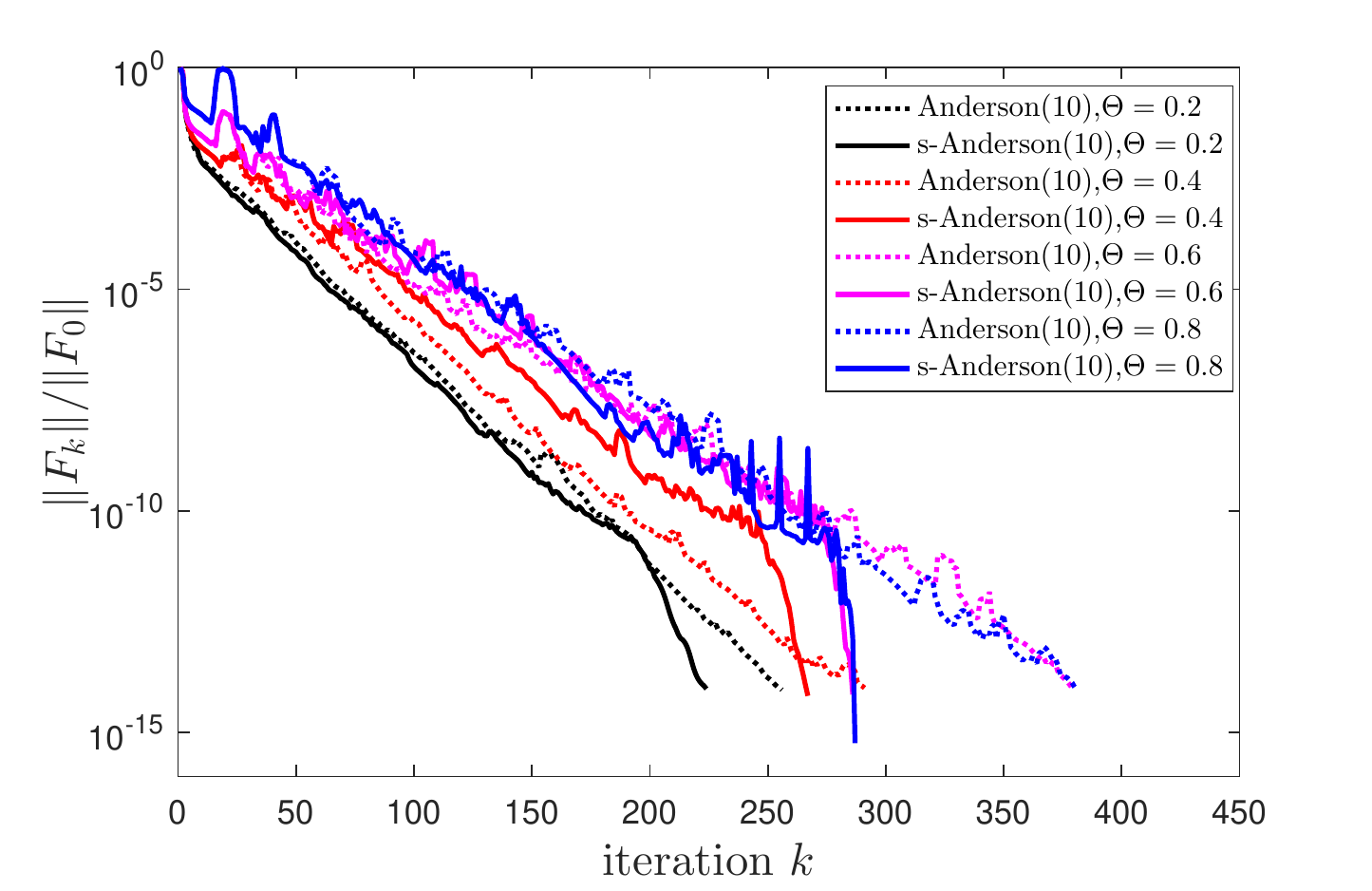}}
\caption{Performance of Anderson(m) and s-Anderson(m) with $m=1,10$ for Example \ref{example2} with four different values of $\Theta$} \label{fig4}
\end{center}
\end{figure}
\end{example}
\subsection{Nonsmooth Dirichlet problem}
Consider the Dirichlet problem \cite{chennashed}
\begin{equation}\label{Dirichlet}
\left\{ \begin{array}{rl}
-\Delta v +\beta v&= \lambda \max\{v-\varphi(x,y), 0\} +\psi(x,y)  \quad {\rm in} \quad \Xi\\
v&=f(x,y)    \quad \quad{\rm on} \quad \bar{\Xi},
\end{array}\right.
\end{equation}
where $\Xi=(0,1)\times(0,1)$, $\bar{\Xi}$ denotes the boundary of $\Xi$,
$\varphi, \psi\in C(\bar{\Xi})\cap C^1(\Xi)$, $f \in C(\bar{\Xi})$, $\beta>0$ and $\lambda\in\mathbb{R}$.
Using the five point centered finite difference method for the Dirichlet problem (\ref{Dirichlet})
with a mesh size $h$ at grid $(x_i,y_j)$ gives
\begin{equation}\label{D1}
 -v_{i,j+1}-v_{i,j-1}+4v_{i,j}-v_{i+1,j}-v_{i-1,j}+\beta h^2 v_{i,j}=\lambda h^2\max\{v_{i,j}-\varphi_{i,j},0\} + h^2 \psi_{i,j}.
 \end{equation}
By transforming $(v_{i,j})$ to a vector $u$, \eqref{D1} can be illustrated by the following system
\begin{equation}\label{Dirichlet0}
(-L+4I-U+\beta h^2I)u=\lambda h^2\max\{u+p,0\} + q,
\end{equation}
where $L$ and $U$ are lower and upper diagonal matrices with nonnegative elements, $h=1/(\sqrt{n}+1)$, $p, q \in \mathbb{R}^n$ are the corresponding vectors transformed by $\varphi_{i,j}$ and $h^2 \psi_{i,j}$. Then, \eqref{Dirichlet0} is equivalent to the following fixed point problem
\begin{equation}\label{DirichletN}
u=G(u):=\frac{1}{4+\beta h^2}(L+U)u+\frac{h^2\lambda}{4+\beta h^2}\max\{u +p, 0\} +\frac{1}{4+\beta h^2} q.
\end{equation}
When $\beta>|\lambda|$, from
$$\|G(u)- G(v)\|\le\frac{4+|\lambda| h^2}{4+\beta h^2}\|u-v\|,$$
the function $G$ in \eqref{DirichletN} is a contraction mapping
with factor
$c:=\frac{4+|\lambda| h^2}{4+\beta h^2}$.

\begin{example}\label{example3}
We consider the nonsmooth fixed point problem (\ref{DirichletN}) from the finite difference discretization of
the nonsmooth Dirichlet problem (\ref{Dirichlet}).
Let the solution of problem (\ref{Dirichlet}) be $v(x,y)=\max(-\sin(x\pi)\sin(y\pi)+0.5,0)$, and $u^*$ present the values of $v(x,y)$ at the mesh points for given mesh size $h=1/(\sqrt{n}+1)$. We randomly generate $p=-0.4*{\rm rand(n,1)}$ and set $q=(4+\beta h^2)u^*-(L+U)u^*- \lambda h^2\max\{u^*+p,0\}$ with $\lambda=1$ and $\beta=2$. Notice that the contraction factor of $G$ is very close to $1$ at this situation.

When $n=64\times 64$, the original function is plotted in Fig. \ref{fig6-2-3}, in which we can see that it is nonsmooth. Choosing the initial point $u_0=0.5*{\rm rand}(n,1)$,
the convergence performance of $\|F_k\|/\|F_0\|$ for s-Anderson(m) are plotted in Fig. \ref{fig6-1-3}.  For different values of $n$, the convergence rates at the stopped point, defined by $(\|F_k\|/\|F_0\|)^{1/k}$, are listed in Table \ref{table1}. This example shows that s-Anderson(m) can effectively solve this problem with a contraction factor very close to $1$, and s-Anderson(m) is faster as $m$ increases from $0$ to $20$.
\begin{figure}
\begin{center}
  \subfigure[$v(x,y)=\max(-\sin(x\pi)\sin(y\pi)+0.5,0)$]{
    \label{fig6-2-3}
    \includegraphics[width=2.3in,height=1.8in]{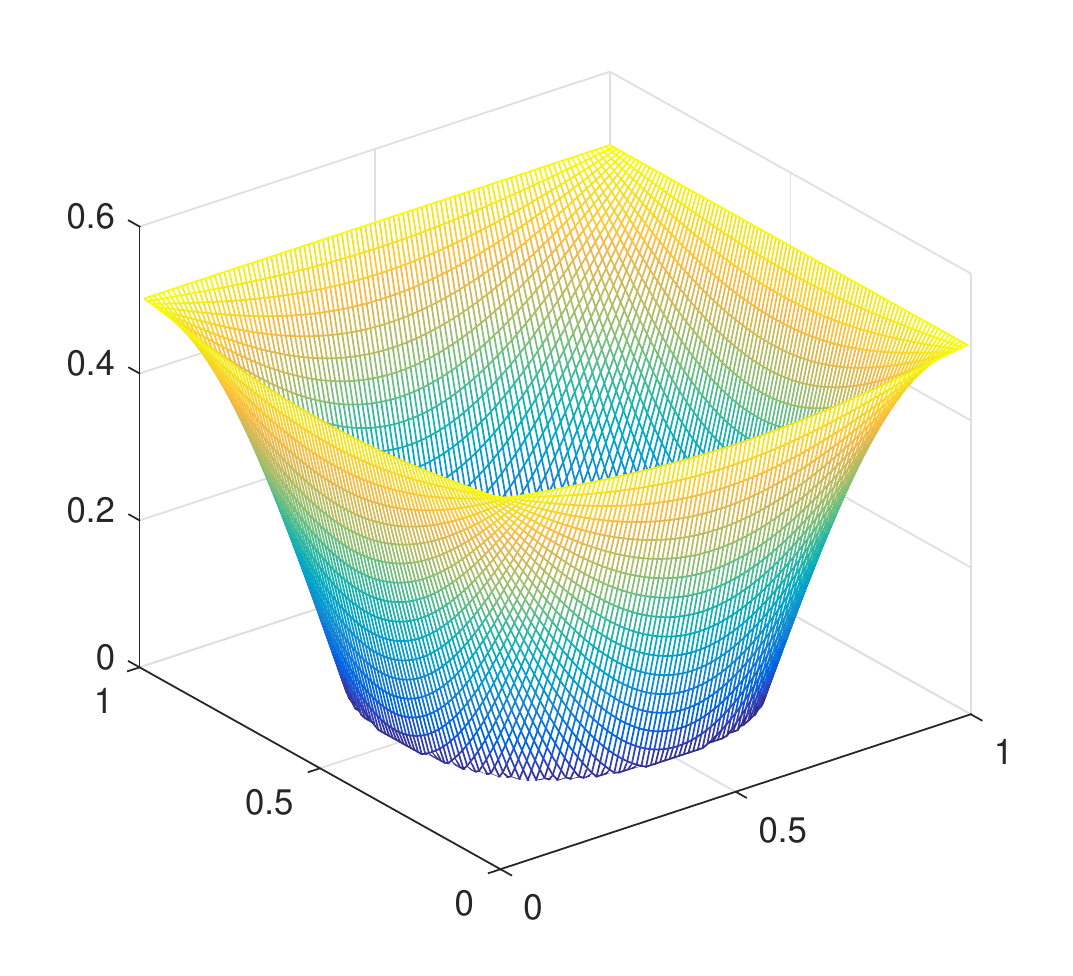}}
\subfigure[Convergence rate]{
    \label{fig6-1-3}
  \includegraphics[width=2.3in,height=1.8in]{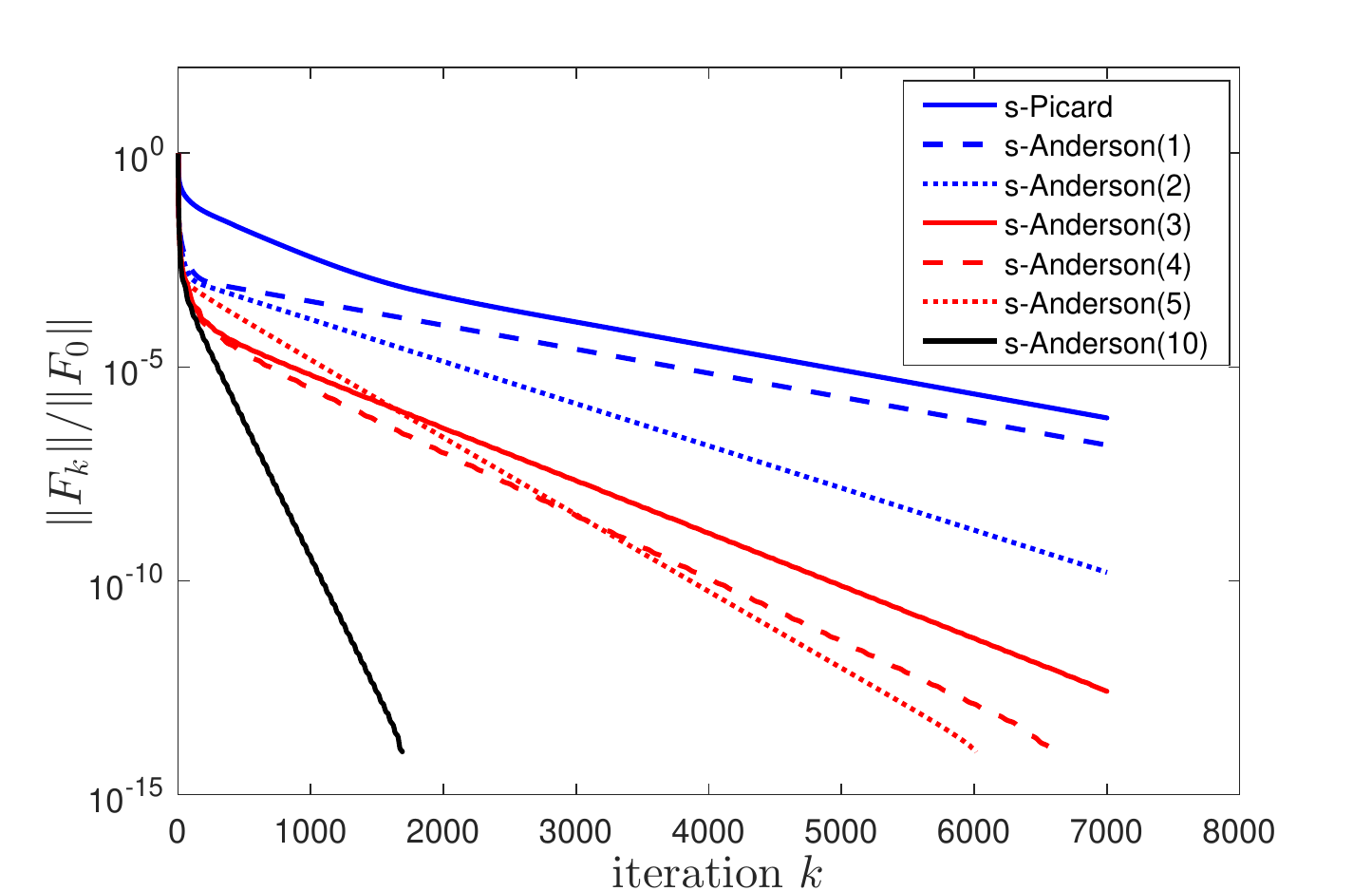}}
\caption{Solution and convergence performance of s-Anderson(m) for Example \ref{example3} with $n=64\times 64$} \label{fig6-3}
\end{center}
\end{figure}
\begin{table}[htbp]
\centering{
{\footnotesize\begin{tabular}{|c|c||c|c|c|c|c|c|c|c|} \hline
$\sqrt{n}$&$1-c$&$m=0$&$m=1$&$m=2$&$m=3$&$m=5$&$m=10$&$m=20$\\
\hline
16&8.635e-4 & 9.793e-01 & 9.789e-01 & 9.632e-01 & 9.519e-01   & 9.191e-01 & 8.501e-01 & 7.750e-01 \\
32&2.294e-4 & 9.944e-01 & 9.940e-01 & 9.897e-01 & 9.860e-01  & 9.784e-01 & 9.480e-01 & 9.078e-01  \\
64&5.916e-5 & 9.978e-01 & 9.978e-01 & 9.968e-01 & 9.959e-01  & 9.946e-01 & 9.832e-01 & 9.758e-01\\
128&1.502e-5 & 9.991e-01 & 9.985e-01 & 9.983e-01 & 9.976e-01 & 9.976e-01 & 9.964e-01 & 9.930e-01 \\
\hline
\end{tabular}}}\caption{Values of $(\|F_k\|/\|F_0\|)^{1/k}$ by s-Anderson(m) for Example \ref{example3}}
\label{table1}
\end{table}
\end{example}
\section{Conclusions}

Anderson acceleration does not use derivatives in its iterations, but it is difficult to prove its convergence without continuous differentiability. Most existing convergence results of Anderson acceleration are established under the assumption that the involved function
is continuously  differentiable \cite{ChenKelley2015,Evans,Pollock,TothKelley2015,Walker}.
For a special class of  nonsmooth functions that is a sum of a smooth term and a nonsmooth term with a small Lipschitz constant, convergence of Anderson acceleration is proved in a recent paper \cite{BCK}. In this paper, we give new convergence results of  Anderson acceleration for
nonsmooth fixed point problem (\ref{eq1-0}), which has a composite max function in $G$.
Theorem \ref{theorem2} shows that Anderson(1) is q-linear convergent with a
q-factor $\hat{c}\in (\frac{2c-c^2}{1-c}, 1)$, which can be strictly smaller than $\frac{3c-c^2}{1-c}$ given
in \cite{BCK,TothKelley2015}.
Moreover, we construct a smoothing approximation ${\cal G}(\cdot,\mu)$ for the nonsmooth function $G$ in (\ref{G-s}), where ${\cal G}(\cdot,\mu)$ is also a contraction mapping and has
the same fixed point as $G$. Then, we propose an Anderson accelerated algorithm with ${\cal G}(u,\mu)$ and prove its local r-linear convergence with factor $c$ for nonsmooth fixed point problem (\ref{eq1-0}),
which is same as the convergence rate
of Anderson acceleration for the continuously differentiable case.

\vspace{0.05in}
{\bf Acknowledgment} We would like to thank Prof. Tim Kelley for introducing us the research topics on Anderson acceleration for nonsmooth fixed problems. We also would like to thank Prof Angela Kunoth, the editors and the two referees for their helpful comments.
\bibliographystyle{siamplain}

\end{document}